\documentclass{article}

\usepackage{amssymb}
\usepackage{amsmath}
\usepackage{bbm}

\usepackage{amsthm}
\newtheorem{lemma}{Lemma}
\newtheorem{theorem}{Theorem}
\newtheorem{proposition}{Proposition}

\theoremstyle{definition}

\newtheorem{remark}{Remark}
\newtheorem{algorithm}{Algorithm}

\usepackage[noblocks]{authblk}

\usepackage{graphicx}
\graphicspath{{images/}}
\usepackage[captionskip=3mm]{subfig}
\usepackage{epstopdf}
\usepackage{float}

\usepackage{booktabs}
\usepackage{microtype}

\usepackage{placeins}

\usepackage[hyperfootnotes=false,linktocpage,pdftex,colorlinks=true,urlcolor=blue,citecolor=blue,linkcolor=red,unicode]{hyperref}

\usepackage{verbatim}

\DeclareMathOperator{\D}{\operatorname{D}}
\DeclareMathOperator{\E}{\mathcal{E}}

\DeclareMathOperator{\F}{\mathcal{F}}

\DeclareMathOperator{\M}{\mathcal{M}}

\DeclareMathOperator{\SBV}{\mathcal{SBV}}
\DeclareMathOperator{\BV}{\mathcal{BV}}
\DeclareMathOperator{\Id}{\operatorname{Id}}

\DeclareMathOperator{\argmin}{arg\,min}

\DeclareMathOperator{\glim}{\operatorname{\Gamma-lim}}
\renewcommand{\L}{\mathcal{L}}
\renewcommand{\H}{\mathcal{H}}
\newcommand{\norm}[1]{\left\| #1 \right\|}
\newcommand{\iprod}[1]{\left\langle #1 \right\rangle}

\numberwithin{equation}{section}
\begin{document}

\title{A variational method for quantitative photoacoustic tomography with piecewise constant coefficients}

\date{October 5, 2015}

\author[1]{E.~Beretta ({elena.beretta@polimi.it})} 
\author[2]{M.~Muszkieta ({monika.muszkieta@pwr.edu.pl})}
\author[3]{W.~Naetar ({wolf.naetar@univie.ac.at})}
\author[3,4]{O.~Scherzer ({otmar.scherzer@univie.ac.at})}

\affil[1]{Dipartimento di Matematica \authorcr Politecnico di Milano \authorcr Via Edoardo Bonardi, 20133 Milan, Italy \vspace{.5\baselineskip}}
\affil[2]{Faculty of Pure and Applied Mathematics  \authorcr Wroclaw University of Technology \authorcr Wyb. Wyspianskiego 27, PL-50-370 Wroclaw, Poland \vspace{.5\baselineskip}}
\affil[3]{Computational Science Center \authorcr University of Vienna \authorcr Oskar Morgenstern-Platz 1, A-1090 Vienna, Austria \vspace{.5\baselineskip}}
\affil[4]{Radon Institute of Computational and Applied Mathematics \authorcr Austrian Academy of Sciences \authorcr Altenbergerstr. 69, A-4040 Linz, Austria \vspace{.5\baselineskip}}

\maketitle

\begin{abstract}
We consider the inverse problem of determining spatially heterogeneous \emph{absorption} and \emph{diffusion} coefficients $\mu(x),D(x)$, from \emph{a single measurement} of the \emph{absorbed energy} $\E(x)=\mu(x) {u}(x)$, where ${u}$ satisfies the elliptic partial differential equation
\begin{equation*}
	-\nabla \cdot (D(x) \nabla {u}(x)) + \mu(x) {u}(x) = 0 \quad \text{in $\Omega \subset \mathbb{R}^N$}.
\end{equation*}
This problem, which is central in \emph{quantitative photoacoustic tomography}, is in general ill-posed since it admits an infinite number of solution pairs. Using similar ideas as in \cite{NaeSch14}, we show that when the coefficients $\mu,D$ are known to be piecewise constant functions, a unique solution can be obtained. For the numerical determination of $\mu,D$, we suggest a variational method based on an \emph{Ambrosio-Tortorelli approximation} of a \emph{Mumford-Shah-like functional}, which we implement numerically and test on simulated two-dimensional data. 

\end{abstract}



\section{Quantitative photoacoustic tomography}
\label{sec:introduction}

\subsection{Introduction}

\emph{Photoacoustic tomography} (PAT) is a medical imaging technique that combines electromagnetic excitation (in the visible spectrum) with ultrasound measurements. In a photoacoustic experiment, a translucent sample is illuminated by a laser pulse with wavelength $\lambda$ (which we assume to be fixed in this article). The absorbed optical energy leads to thermal expansion, generating an ultrasound pressure wave $p(x,t)$ that can be measured outside the sample.
\FloatBarrier 

Since the pulses used in PAT are very short, the complete energy is deposited almost instantaneously compared to travel times of acoustic waves, and the pressure wave $p$ can be assumed to have been generated by an \emph{initial pressure} $\H(x)$, that is, it satisfies the wave equation 
\begin{equation*}
 \begin{aligned}
 \partial_{tt}p(x,t) - c^2(x) \Delta p(x,t) &= 0 \\
	\partial_t p(x,0) &= 0 \\
	p(x,0) &= \H(x)
 \end{aligned}
\end{equation*}
and $p|_{\M}$, where $\M$ denotes a \emph{measurement surface}, can be obtained from ultrasound measurements \cite{CoxLauArrBea12,CoxLauBea09,KucKun08,WanWu07}. 

By solving an \emph{inverse problem for the wave equation} (see, e.g, \cite{KucKun08} for a review of inversion techniques), these ultrasound measurements can be used to estimate the \emph{initial pressure} $\H(x)$.  Since 
\begin{equation}
\label{eq:def_E}
	\H(x) = \Gamma(x) \E(x) = \Gamma(x) \mu(x) {u}(x),
\end{equation}
that is, $\H(x)$ is proportional to the \emph{absorbed energy} $\E(x)$, which is in turn proportional to the \emph{optical absorption coefficient} $\mu(x)$ at the applied wavelength $\lambda$ and the local fluence $u(x)$ (the time-integrated laser power received at $x$), the initial pressure visualizes contrast in $\mu$. The constant of proportionality $\Gamma(x)$ is called \emph{Gr\"uneisen coefficient} (or \emph{PA efficiency} since it describes the efficiency of conversion from absorbed energy to acoustic signal) \cite{CoxLauArrBea12,CoxLauBea09}.

In \emph{quantitative photoacoustic tomography} (qPAT), the goal is to apply PAT to determine (inhomogeneous) optical material properties of the sample (which are of diagnostic interest). To do so, an additional non-linear \emph{inverse problem for light transport} has to be solved, since the fluence $u(x)$ is inhomogeneous and itself dependent on the optical properties of the sample \cite{CoxLauArrBea12,CoxLauBea09}.
 
While it is also possible to attack the acoustic and optical inverse problems simultaneously (see, e.g., \cite{BerBonHabPri14,HalNeuRab15}), here, we assume that the acoustic part of the problem has been solved successfully, i.e., that (possibly noisy) data $\H^\delta \approx \H$ (with $\delta$ signifying the noise level) are available. 

In this article, we utilize the \emph{diffusion approximation} (which is valid in highly scattering media \cite{CoxLauArrBea12,WanWu07}) of the \emph{radiative transfer equation} to model the fluence distribution. It is, however, also possible to use a radiative transfer model for qPAT (see \cite{DeCTraSej15,PulCoxArrKaiTar15,SarTarCoxArr13,TarCoxKaiArr12,YaoSunJia09}), albeit at the cost of increased computational and analytical complexity. 

The diffusion model (in a Lipschitz domain $\Omega \subset \mathbb{R}^N$) is given by
\begin{equation}
\label{eq:diff_eq_dirichlet}
 \begin{aligned}
	-\nabla \cdot (D(x) \nabla {u}(x)) + \mu(x) {u}(x) &= 0 \quad \text{in $\Omega$}\\
	{u}(x)|_{\partial \Omega} &= g(x).
 \end{aligned}	
\end{equation}
The parameter $D$ is called \emph{diffusion coefficient}. The Dirichlet boundary data $g$ (which we assume to be continuous and known in this article) describes the illumination pattern. Note that this is a time-independent model, again due to the fact that energy is deposited almost instantaneously compared to time scales of the acoustic system. 

The difficulty of the inverse problem varies depending on which parameters of the model are assumed to be known. We present three inverse problems often considered in the literature. The hardest one is: 
\begin{equation}
\label{prob:problem3}
	\text{ Determine $(\mu,D,\Gamma)$ from measurements of $\H^\delta$. }	
	\tag{P3}
\end{equation}

Bal and Ren showed in \cite{BalRen11a} that for arbitrary coefficients $\mu,D,\Gamma \in W^{1,\infty}(\Omega)$, this problem is unsolvable, even if multiple measurements of $\H$ (with different known boundary illumination patterns $g$) are available. 

In \cite{NaeSch14}, the authors showed that with a restriction to \emph{piecewise constant parameters}, unique reconstruction of all three unknown parameters $\mu,D,\Gamma$ from multiple measurements is possible (under a condition on the directions of $\nabla {u_k}$, where $u_k$ is the fluence of the $k$th illumination pattern). Furthermore, an analytical reconstruction procedure was suggested and implemented numerically, which, unfortunately, is relatively sensitive to noise.

Alberti and Ammari \cite{AlbAmm15} also established a unique reconstruction result, based on \emph{morphological component analysis} (a sparsity approach), in a slightly more general setting (which assumes different degrees of smoothness of the coefficients and the fluence). They also provide numerical reconstructions, which were, however, not tested for noise sensitivity in the case of \eqref{prob:problem3}.

To simplify the problem, it is often assumed that the Gr\"uneisen coefficient $\Gamma$ is known or constant, which implies that the absorbed energy can be estimated with $\E^\delta = \frac{\H^\delta}{\Gamma}\approx \E$ (with $\delta$ again denoting the noise level). It remains to solve
\begin{equation}
\label{prob:problem2}
	\text{ Determine $(\mu,D)$ from measurements of $\E^\delta$. }	
	\tag{P2}
\end{equation}
If only a single measurement of $\E^\delta$ is given, this inverse problem is also ill-posed (since it has infinitely many solutions pairs, see \cite{CoxArrBea09,NaeSch14,ShaCoxZem11}).

However, in \cite{AlbAmm15}, the authors were able to recover $\mu$ (independently of the light transfer model used) from a single measurement of $\E$ (again using a \emph{sparsity method}, assuming different degrees of smoothness of the coefficients and the fluence).

In \cite{BalRen11a,BalUhl10} it was shown that this problem is uniquely solvable if two measurements of $\E$ (corresponding to well-chosen boundary illuminations $g_1,g_2$) are available. Numerically, this multi-illumination case was treated in \cite{BalRen11a,GaoOshZha12,RenGaoZhao13,ShaCoxZem11,TarCoxKaiArr12,Zem10}, using a multitude of different techniques, see also the review paper by Cox et al.  \cite{CoxLauArrBea12}.

The simplest case works under the assumption that the diffusion coefficient $D$ is also known:
\begin{equation}
\label{prob:problem1}
	\text{ Determine $\mu$ from measurements of $\E^\delta$. }	
	\tag{P1}
\end{equation}
This inverse problem has a unique solution even for a single measurement, which can be seen by substituting $\mu u = \E^\delta$ in \eqref{eq:diff_eq_dirichlet}, providing the possibility to solve for $u$ \cite{BanBagVasRoy08}. For other (numerical) approaches, cf. \cite{CoxLauArrBea12}. To the authors knowledge, this simplified problem is the only case for which practical viability (with experimental data) has been established both for phantoms \cite{YuaJia06} and biological samples \cite{SunSobJia09,SunSobJia13}. 

\subsection{Contributions of this article}

In this article, we consider the problem \eqref{prob:problem2} using a \emph{single measurement} for our reconstructions, a problem that is in general ill-posed. Similar to \cite{NaeSch14} (which treats \eqref{prob:problem3} with multiple measurements), a restriction to piecewise constant $\mu,D$ also proves to be useful for this problem. In fact, as shown in Section \ref{sec:recovery_pw_const}, when the parameters $\mu,D$ are piecewise constant functions (and noise-free data are given), the inverse problem \eqref{prob:problem2} can be solved uniquely, without any further assumptions. 

In Section \ref{sec:mumford_shah_parameter_detection}, we present a variational model for the reconstruction of piecewise constant $\mu,D$ from noisy data $\E^\delta$ based on the \emph{Ambrosio-Tortorelli approximation} of a \emph{Mumford-Shah-like functional}. Compared to the two-step reconstruction process presented in \cite{NaeSch14} (which was introduced for \eqref{prob:problem3}, but is applicable for \eqref{prob:problem2}), which first detects the regions where the parameters are constant, then reconstructs the parameter values from jumps of the data and its derivatives, this variational approach is much more robust with respect to noise. This is mainly due to the fact that the numerical approach presented in \cite{NaeSch14} requires almost perfect jump detection in the second derivatives of $\E^\delta$ to get reasonable estimates of $\mu,D$ (since the jumps have to form a full partition of the domain $\Omega$), which is highly challenging in the presence of significant amounts of noise.  This is not the case for the variational method presented here.

Finally, a description of our implementation and numerical results can be found in Section \ref{sec:numerical_results}.

\section{Recovery of piecewise constant coefficients}
\label{sec:recovery_pw_const}

In this Section we show that piecewise constant parameters $\mu,D$ can be recovered uniquely from a single measurement of the absorbed energy $\E(\mu,D)$.

In the following, let $(\Omega_m)_{m=1}^M$ be a partition of $\Omega \subset \mathbb{R}^N$ into open sets and $\mu,D$ piecewise constant on $(\Omega_m)_{m=1}^M$, that is, for $\Omega_m$ open and $\mu_m,D_m \in \mathbb{R}^+$ 
\begin{equation}
\label{eq:pw_const}
	\overline\Omega = \bigcup_{m=1}^M \overline\Omega_m, \enskip \mu = \sum_{m=1}^M \mu_m 1_{\Omega_m}, \enskip D = \sum_{m=1}^M D_m 1_{\Omega_m}.
\end{equation}
Furthermore, for $k \in \mathbb{N}$, denote by 
\begin{equation*}
J_k(f)=\Omega \setminus \bigcup \{ B \subset \Omega \mid B \text{ is open and } f \in C^k(B) \}
\end{equation*}
the discontinuities of a function $f \in L^\infty(\Omega)$ and its derivatives up to $k$-th order.

First, we show that coefficient discontinuities can be recovered from the data $\E$.

\begin{proposition}
\label{prop:pw_const_prop_1}
Let $\mu,D$ be of the form \eqref{eq:pw_const} and $\E=\E(\mu,D)$ satisfy \eqref{eq:def_E},\eqref{eq:diff_eq_dirichlet} weakly. Then we have
\begin{equation}
\label{eq:prop_jumps_D2}
	J_0(\mu) \cup J_0(D) = J_2(\E).
\end{equation}
\end{proposition}

\begin{proof}
First, take an arbitrary open ball $B$ with $B \cap (J_0(\mu) \cup J_0(D)) = \emptyset$ and notice that, by interior elliptic regularity, ${u} \in C^\infty(B)$ and hence also $\E \in C^\infty(B)$. It follows that
\begin{equation*}
	J_0(\mu) \cup J_0(D) \supset J_2(\E).
\end{equation*}
By the De Giorgi-Nash-Moser theorem \cite[Theorem 8.22]{GilTru01}, we have ${u} \in C^0(\overline\Omega)$, so 
\begin{equation}
\label{eq:mu_data_jump}
	J_0(\mu) = J_0(\E).
\end{equation}
Finally, as $u \in C^\infty(\Omega_m), \, m=1,\ldots,M$, \eqref{eq:diff_eq_dirichlet} holds strongly in all $\Omega_m$, and we get 
\begin{equation}
\label{eq:diff_eq_scalar}
\begin{aligned}
	-D_m \Delta {u} + \mu_m {u} &= 0  \quad \text{in $\Omega_m$} \\
	-D_m \Delta \E + \mu_m \E &= 0 \quad \text{in $\Omega_m$} 	
\end{aligned}
\end{equation}
since $\mu_m,D_m$ are scalars. Hence, $\Delta \E = \frac{\mu}{D} \E = \frac{\mu^2}{D} {u}$ in $\bigcup_{m=1}^M \Omega_m$, which shows that $\Delta \E$ cannot be continuous where $\frac{\mu^2}{D}$ jumps (since ${u} \in C^0(\overline\Omega)$), hence 
\begin{equation*}
	J_0\left(\frac{\mu^2}{D}\right) \subset J_2(\E).
\end{equation*}
The Proposition follows from $J_0(\mu) \cup J_0\left(\frac{\mu^2}{D}\right) = J_0(\mu) \cup J_0(D)$.
\end{proof}

\begin{remark}
\label{rem:edges_first_order}
The proof of Proposition \ref{prop:pw_const_prop_1} also shows that jumps in $\mu$ and $D$ have different effects on the data $\E$. By \eqref{eq:mu_data_jump}, the jumps 	 of $\mu$ can be recovered from those in $\E$. Jumps of $D$ (not coinciding with jumps of $\mu$) on the other hand, are smoothed in the data and have to be obtained from the derivatives of $\E$. In particular, the discontinuities of $\Delta \E$ suffice to find the jumps in $D$.

Furthermore, note that from discontinuities of $\nabla \E$ (or, equivalently, $|\nabla \E|^2$), one can, in general, only recover the part of the jumps of $D$ where $\nabla {u}(x) \cdot \nu(x) \neq 0$ (where $\nu$ denotes a normal vector to the hyper-surface of discontinuity in $D$), as $\nabla {u}$ may be continuous across parts of the jump set of $D$ where $\nabla {u}(x)$ is tangential. 

To see this, consider for instance the case of a partition consisting of a homogeneous region $\Omega_1$ with (non-touching) smooth inclusions $\Omega_2,\ldots,\Omega_M$. In this case, we have ${u} \in C^\infty(\overline\Omega_m)$ for all $m=1,\ldots,M$ (due a result of Li and Nirenberg, see \cite{LiNir03}) and the interface conditions 
\begin{equation}
\label{eq:interface_cond}
	D_m \nabla {u}|_{\Omega_m} \cdot \nu = D_1 \nabla {u}|_{\Omega_1} \cdot \nu, \quad
	\nabla {u}|_{\Omega_m} \cdot \tau =  \nabla {u}|_{\Omega_1} \cdot \tau \quad 
\end{equation}
hold point-wise on $\partial\Omega_m$ for $m=2,\ldots,M$, where $\nu,\tau$ denote vectors normal and tangential to $\partial\Omega_m$ (see, e.g., \cite{NaeSch14} for a derivation).  Now, let 
\begin{equation*}
\Lambda:=\{x \in \bigcup_{m=2}^M \partial\Omega_m \mid \nabla {u}(x) \cdot \nu(x) \neq 0\}.
\end{equation*}
We have $J_0(D) \cap \Lambda \subset J_1({u})$ since, by the interface conditions \eqref{eq:interface_cond}, $\nabla {u}$ cannot be continuous across $J_0(D) \cap \Lambda$ (it changes length). On the other hand, if we take an open ball $B \subset \Omega_k \cup \Omega_1$ (for some $2 \leq k \leq M$) with $B \cap J_0(D) \cap \Lambda = \emptyset$, we have, again by \eqref{eq:interface_cond},
\begin{equation}
\label{eq:nabla_phi_cont}
	\nabla {u}|_{\Omega_k}=\nabla {u}|_{\Omega_1} \quad \text{ on $\partial\Omega_k \cap B$}
\end{equation}
since either $D_k=D_1$ or $\nabla {u}|_{\Omega_k \cap B} \cdot \nu = \nabla {u}|_{\Omega_1 \cap B} \cdot \nu = 0$. As we have ${u} \in C^\infty(\overline\Omega_k)$, ${u} \in C^\infty(\overline\Omega_1)$, \eqref{eq:nabla_phi_cont} implies ${u} \in C^1(B)$ and therefore $J_1({u}) \subset \overline{J_0(D) \cap \Lambda}$, hence
\begin{equation*}
	\overline{J_1({u})} = \overline{J_0(D) \cap \Lambda}.
\end{equation*}
This shows that, in general, discontinuities in the second derivatives of the data $\E$ have to be identified in order to get the whole jump set of $D$.
\end{remark}

\vspace\baselineskip

Once the partition $(\Omega_m)_{m=1}^M$ is known, the coefficients $\mu,D$ can be recovered from the jumps in $\E, \frac{\Delta \E}{\E}$ and the boundary values of ${u}$.

\begin{proposition}
\label{prop:pw_const_prop_2}
Let $\mu,D$ be of the form \eqref{eq:pw_const} with $(\Omega_m)_{m=1}^M$ known. Furthermore, let $\E$ satisfy \eqref{eq:def_E},\eqref{eq:diff_eq_dirichlet} for a known boundary illumination $g$. Then, $\mu$ and $D$ can be determined uniquely from $\E$.
\end{proposition}

\begin{proof}
Let $x \in \partial\Omega_m \cap \partial\Omega_n$ for some $m,n \in \{1,\ldots,M\}$. Since ${u} \in C^0(\overline\Omega)$, 
\begin{equation}
\label{eq:mu_jump}
	\frac{\lim_{y \to x} \E(y)|_{\Omega_m}}{\lim_{z \to x} \E(z)|_{\Omega_n}}= \frac{\mu_m {u}(x)}{\mu_n {u}(x)} =  \frac{\mu_m}{\mu_n}.
\end{equation}
Furthermore, take $x \in \partial\Omega \cap \partial\Omega_k$ (for some $k \in \{1,\ldots,M\}$). We have
\begin{equation*}
	\lim_{y \to x} \frac{\E(y)|_{\Omega_k}}{g(x)}=\mu_k \lim_{y \to x} \frac{{u}(y)|_{\Omega_k}}{g(x)} = \mu_k.
\end{equation*}
Starting in $\Omega_k$ and using \eqref{eq:mu_jump} on all interfaces, we recover all $\mu_m, \,m=1,\ldots,M$.

Finally, to obtain $D_m, \, m=1,\ldots,M$, we use \eqref{eq:diff_eq_scalar}, that is,
\begin{equation*}
\label{eq:mu_D_ratio}
	\frac{\mu_m}{D_m}=\frac{\Delta \E(y)}{\E(y)} \quad \text{for all $y \in \Omega_m$}.
\end{equation*}
\end{proof}

\begin{remark}
If $g$ is not known, Proposition \ref{prop:pw_const_prop_2} can be used to determine the parameters $\mu,D$ up to a constant.
\end{remark}

\section{A Mumford-Shah-like functional for qPAT}
\label{sec:mumford_shah_parameter_detection}

In Section \ref{sec:recovery_pw_const}, we showed that piecewise constant absorption and diffusion coefficients $\mu,D$, can be recovered from noise-free data $\E$ by an analytical procedure which first determines the coefficient jumps, i.e., the partition $(\Omega_m)_{m=1}^M$, and then the numerical values $(\mu_m,D_m)_{m=1}^M$. In \cite{NaeSch14}, such a two-step approach was implemented numerically (for the problem \eqref{prob:problem3}). 

In the presence of significant amounts of noise in the data, however, such an approach is infeasible, since it requires the detection of jumps of derivatives up to second order of the data $\E^\delta$. In particular, jumps may remain partially undetected (e.g., if the edge detection has to be restricted to first derivatives due to noise, see Remark \ref{rem:edges_first_order} and the numerical examples in Section \ref{sec:numerical_results}), leading to an incomplete estimated partition and therefore highly erroneous parameter estimates. 

To overcome this problem, we propose a variational approach favouring piecewise constant solutions that estimates the numerical values of piecewise constant $\mu,D$ and their jumps at the same time. 

There are multiple different methods for piecewise constant regularization of inverse problems. For instance, a popular class of methods is based on the \emph{level set method} \cite{OshSet88} or variations of it, e.g., \cite{BurOsh05,ChaVes01,DeCLeiTai09,TaiCha04,TaiLi06}). This approach has been suggested for qPAT (using the radiative transfer model) in \cite{DeCTraSej15}, however, it was only tested using multiple measurements. In this article, we use an \emph{Ambrosio-Tortorelli} approximation of a \emph{Mumford-Shah-like} functional (which was first suggested for \emph{electrical impedance tomography} in \cite{RonSan01}). The main advantage of this approximation is that we can utilize incomplete jump information (obtained, for instance, from jumps of the data or other means) to initialize the minimization procedure (see Section \ref{sec:mumford_shah_minimization}). Numerically, leads to faster convergence (and in some cases improved minimizers). Additionally, the number of segments does not have to be known in advance (in contrast to \emph{multiple level set methods}) and the minimization with respect to the jump indicator functions is a simple elliptic problem. 

We want to minimize the Mumford-Shah-like functional 
\begin{equation}
\label{eq:qpat_ms_funct}
 \begin{aligned}
	\F(\mu,D,K_\mu,K_D) &= \frac{1}{2} \norm{\E(\mu,D)-\E^\delta}_{L^2(\Omega)}^2 \\
	&+ \frac{\alpha_\mu}{2} \int_{\Omega \setminus K_\mu} |\nabla \mu|^2 \,dx  + \frac{\alpha_D}{2} \int_{\Omega \setminus K_D} |\nabla D|^2 \,dx \\ 
	&+ \beta_\mu H^{N-1}(K_\mu) + \beta_D H^{N-1}(K_D),
 \end{aligned}	
\end{equation}
where $H^{N-1}$ is the $(N-1)$-dimensional Hausdorff-measure and $\E(\mu,D)$ the operator that maps the (unknown) parameters $\mu,D$ to the measurements $\E$ satisfying \eqref{eq:def_E},\eqref{eq:diff_eq_dirichlet}. The minimum is taken over all $\mu,D$ in suitable (that is, point-wise bounded from below and above) subsets of $W^{1,2}(\Omega \setminus K_\mu), W^{1,2}(\Omega \setminus K_D)$ and all closed sets $K_\mu,K_D \subset \Omega$ (the jump sets of the coefficients $\mu,D$). 

Functionals of this type, first introduced by Mumford and Shah for image denoising and segmentation in \cite{MumSha89}, have been applied for a wide range of inverse problems, see, e.g., \cite{EseShe02,Kla11,KreRie14,RamRin07,RamRin10,RonSan01}). The basic idea behind the functional is as follows. The \emph{discrepancy term} 
\begin{equation*}
\frac{1}{2} \norm{\E(\mu,D)-\E^\delta}_{L^2(\Omega)}^2
\end{equation*} 
forces minimizing coefficients $\hat\mu,\hat D$ to be close (in $L^2$-sense) to a solution of the inverse problem $\E(\mu,D)=\E^\delta$ (of which there are infinitely many), while the \emph{regularization terms} 
\begin{equation*}
 \frac{\alpha_\mu}{2} \int_{\Omega \setminus K_\mu} |\nabla \mu|^2 \,dx + \frac{\alpha_D}{2} \int_{\Omega \setminus K_D} |\nabla D|^2 \,dx + \beta_\mu H^{N-1}(K_\mu) + \beta_D H^{N-1}(K_D)
\end{equation*}
force both the variation of $\hat\mu,\hat D$ (outside their jump sets $\hat K_\mu,\hat K_D$) as well as the hyper-surface area of $\hat K_\mu,\hat K_D$ to be small. The regularization parameters $\alpha_\mu,\alpha_D$ control the amount of continuous variation the parameters may have, whereas $\beta_\mu,\beta_D$ control the complexity of the coefficient's jump sets $K_\mu,K_D$. 

Note that while in general, minimizers of this functional can have some continuous variation, they have to be close to piecewise constant for large values of $\alpha_\mu,\alpha_D$. In fact, the Ambrosio-Tortorelli approximation of the functional \eqref{eq:qpat_ms_funct}, which will be introduced in Section \ref{sec:mumford_shah_approximation}, can also be used for piecewise constant regularization (in the limit $\alpha_\mu,\alpha_D \to \infty$, see Remark \ref{rem:piecewise_constant_limit}).

\subsection{Existence of minimizers}
\label{sec:mumford_shah_existence}

Existence of minimizers of functionals like \eqref{eq:qpat_ms_funct} (which are, in general, not unique) was first established in \cite{DeGCarLea89} for image segmentation and in \cite{RonSan01} for regularization of non-linear operator equations. 

To ensure that $\E(\mu,D)$ is well-defined and that $\F$ has a minimizer, point-wise bounds have to be enforced (see \cite{JiaMaaPag14}), so we have to restrict the coefficients to $X_a^b(\Omega) = \{ f \in L^\infty(\Omega)\colon a \leq f \leq b \text{ a.e. in }\Omega\}$ for some $0 < a,b < \infty$ a priori. We also require the following Lemma, the proof follows the ideas in \cite{EggSchl10}. 

\begin{lemma}
\label{lem:continuity_L2_E}
The non-linear measurement operator 
\begin{equation*}
\E\colon X_a^b(\Omega)^2 \subset L^2(\Omega)^2 \to L^2(\Omega), \ \E(\mu,D) = \mu \,{u}(\mu,D),  
\end{equation*}
where ${u}(\mu,D)$ solves \eqref{eq:diff_eq_dirichlet}, is continuous.
\end{lemma}
\begin{proof}
Let $(\mu_n,D_n) \to (\mu,D)$ in $X_a^b(\Omega)^2 \subset L^2(\Omega)^2$. Furthermore, let ${u}_n := {u}(\mu_n,D_n)$ and ${u} := {u}(\mu,D)$ be solutions of \eqref{eq:diff_eq_dirichlet} corresponding to the given coefficients.

By the usual energy estimates for \eqref{eq:diff_eq_dirichlet} (cf. \cite[Chapter 6, Theorem 2]{Eva98}) we have
\begin{equation*}
\norm{{u}_n}_{W^{1,2}(\Omega)} \leq C(a,b,\Omega) \norm{g}_{L^2(\partial\Omega)}.
\end{equation*}
Hence, there exists a (re-labeled) subsequence $({u}_k)_{k \in \mathbb{N}}$ with ${u}_k \rightharpoonup \overline{u}$ weakly in $W^{1,2}(\Omega)$ for some $\overline{u} \in W^{1,2}(\Omega)$. From the weak form of \eqref{eq:diff_eq_dirichlet} we get for all $v \in W^{1,\infty}(\Omega)$ and $k \in \mathbb{N}$
\begin{equation*}
\begin{aligned}
	&\iprod{D \nabla( {u}_k-{u}), \nabla v}_{L^2(\Omega)^2} + \iprod{\mu({u}_k-{u}), v}_{L^2(\Omega)}  \\	
	&\hspace{0.2\linewidth} = \iprod{(D-D_k) \nabla {u}_k, \nabla v}_{L^2(\Omega)^2} + \iprod{(\mu-\mu_k) {u}_k, v}_{L^2(\Omega)} \\
	&\hspace{0.2\linewidth} \leq C(\norm{D_k-D}_{L^2(\Omega)} + \norm{\mu_k-\mu}_{L^2(\Omega)}) \norm{g}_{L^2(\partial\Omega)} \norm{v}_{W^{1,\infty}(\Omega)}.
\end{aligned}	
\end{equation*}
Taking $k \to \infty$ and using the fact that the left side acts as a bounded linear functional on ${u}_k \in W^{1,2}(\Omega)$, we obtain for all $v \in W^{1,\infty}(\Omega)$
\begin{equation*}
 	\iprod{D \nabla(\overline{u}-{u}), \nabla v}_{L^2(\Omega)^2} + \iprod{\mu(\overline{u}-{u}), v}_{L^2(\Omega)} = 0.
\end{equation*}
By density of $W^{1,\infty}(\Omega) \subset W^{1,2}(\Omega)$, this implies $\overline{u}={u}$. Since the same argument also holds for every subsequence of the original sequence $({u}_n)_{n \in \mathbb{N}}$, we get ${u}_n \rightharpoonup {u}$ weakly in  $W^{1,2}(\Omega)$ and thus ${u}_n \to {u}$ strongly in $L^2(\Omega)$ . 

Finally, since ${u} \in L^\infty(\Omega)$ by the maximum principle, continuity of $\E$ follows from
\begin{equation*}
	\norm{\mu_n {u}_n - \mu {u}}_{L^2(\Omega)} \leq \norm{\mu_n-\mu}_{L^2(\Omega)} \norm{{u}}_{L^\infty(\Omega)} + \norm{{u}_n-{u}}_{L^2(\Omega)} \norm{\mu_n}_{L^\infty(\Omega)}.
\end{equation*}
\end{proof}
\begin{remark}
\label{rem:continuity_Lp_E}
Since  $\norm{f}_{L^2}^2 \leq \norm{f}_{L^1(\Omega)} \norm{f}_{L^\infty(\Omega)}$ for all $f \in X_a^b(\Omega)$, 
the operator $\E\colon X_a^b(\Omega)^2 \subset L^1(\Omega)^2 \to L^2(\Omega)$ is also continuous.
\end{remark}

Following \cite{Amb90,RonSan01}, we show the existence of minimizers of a weak form of $\F$ defined for coefficients $\mu,D$ in the space $\SBV(\Omega)$, the \emph{special functions of bounded variation}, which may have jump discontinuities.  For a quick introduction, see Appendix \ref{sec:sbv_introduction}.

We obtain the functional $\overline{\F}\colon X_a^b(\Omega)^2 \cap \SBV(\Omega)^2 \to \mathbb{R}$,  
\begin{equation}
\label{eq:qpat_ms_funct_weak}
\begin{aligned}
	\overline{\F}(\mu,D) &= \frac{1}{2} \norm{\E(\mu,D)-\E^\delta}_{L^2(\Omega)}^2 + \frac{\alpha_\mu}{2} \int_{\Omega} |\nabla \mu|^2 \,dx  \\ 
	 &+ \frac{\alpha_D}{2} \int_{\Omega} |\nabla D|^2 \,dx + \beta_\mu \,H^{N-1}[S(\mu)] + \beta_D \,H^{N-1}[S(D)],
\end{aligned}
\end{equation}
where $S(f)$ denotes the \emph{approximate discontinuity set} of a Lebesgue-measurable function $f$ and $\nabla \mu, \nabla D$ the density of the absolutely continuous part (with respect to the Lebesgue measure) of the respective distributional gradients (see Appendix \ref{sec:sbv_introduction}).

In this setting, the existence of minimizers can be established using the direct method and the $\SBV$ compactness theorem. 
\begin{proposition}
\label{prop:existence_minimizers_ms}
The weak Mumford-Shah-like functional $\overline{\F}$ has at least one minimizer $(\hat \mu, \hat D)$ in $X_a^b(\Omega)^2 \cap \SBV(\Omega)^2$.
\end{proposition}
\begin{proof}
Let $(\mu_n,D_n) \in X_a^b(\Omega)^2 \cap \SBV(\Omega)^2$ be a minimizing sequence of the functional \eqref{eq:qpat_ms_funct_weak}. Clearly, we have for all $n \in \mathbb{N}$
\begin{equation*}
\begin{aligned}
	&\norm{\mu_n}_{L^\infty(\Omega)} + \int_\Omega |\nabla \mu_n|^2 \,dx + H^{N-1}[S(\mu)] \leq C < \infty \\
	&\norm{\D_n}_{L^\infty(\Omega)} + \int_\Omega |\nabla \mu_n|^2 \,dx + H^{N-1}[S(D)] \leq C < \infty.
\end{aligned}	
\end{equation*}
By the $\SBV$-compactness theorem (see Section \ref{sec:sbv_introduction}), and using Lemma \ref{lem:continuity_L2_E}, first applied to $\mu_n$, then to the obtained subsequence of $D_n$, there exist $(\hat \mu,\hat D) \in \SBV(\Omega)^2$  and a subsequence $(\mu_k,D_k)_{k \in \mathbb{N}}$ (re-labeled) with
\begin{equation*}
\begin{aligned}
	(\mu_k,D_k) &\to (\hat\mu,\hat D) \text{ strongly in } L^1_{\text{loc}}(\Omega)^2 \\
	(\nabla \mu_k,\nabla D_k) &\rightharpoonup (\nabla \hat\mu,\nabla \hat D)  \text{ weakly in } L^2(\Omega;\mathbb{R}^N)^2 \\
	H^{N-1}[S(\hat\mu)] &\leq \liminf_{k \to \infty} H^{N-1}[S(\mu_k)] \\
	H^{N-1}[S(\hat D)] &\leq \liminf_{k \to \infty} H^{N-1}[S(D_k)]. 
\end{aligned}	
\end{equation*}
We also have $(\hat\mu,\hat D) \in X_a^b(\Omega)^2$ since $X_a^b(\Omega)^2$ is closed in $L^1_{\text{loc}}(\Omega)^2$. Furthermore, it is a minimizer of $\overline\F$ since the $L^2(\Omega)$-norm is weakly lower semi-continuous and $\E\colon X_a^b(\Omega)^2 \subset L^1(\Omega)^2 \to L^2(\Omega)$ is continuous.
\end{proof}

\begin{remark} 
If the minimizer $(\hat\mu,\hat D)$ has \emph{essentially closed} jump sets $S(\hat\mu),S(\hat D)$ (that is, taking the closure of the jump sets does not increase their $H^{N-1}$-measure) the quadruple $\left( \hat\mu,\hat D,\overline{S(\hat\mu)},\overline{S(\hat D)} \right)$ also minimizes the strong Mumford-Shah functional \eqref{eq:qpat_ms_funct} among all $\mu \in X_a^b(\Omega) \cap W^{1,2}(\Omega \setminus K_\mu)$,  $D \in X_a^b(\Omega) \cap W^{1,2}(\Omega \setminus K_D)$ and $K_\mu, K_D$ closed. While essential closedness of jump sets of Mumford-Shah-minimizers can be established under certain conditions on the discrepancy term \cite{DeGCarLea89,JiaMaaPag14,RonSan01}, this is out of the scope of this article. 
\end{remark}

\subsection{Approximation}
\label{sec:mumford_shah_approximation}

Using the approach of Ambrosio and Tortorelli (cf. \cite{AmbTor92}), one can obtain functionals 
\begin{equation*}
	\F_\epsilon\colon X_a^b(\Omega)^2 \cap W^{1,2}(\Omega)^2 \times X_0^1(\Omega)^2 \cap W^{1,2}(\Omega)^2 \to \mathbb{R}
\end{equation*}
that approximate the functional $\overline\F$ from \eqref{eq:qpat_ms_funct_weak} and are easier to minimize: 
\begin{equation}
\label{eq:qpat_ms_at_funct}
 \begin{aligned}
	\F_\epsilon(\mu,D&,v_\mu,v_D) = \frac{1}{2} \norm{\E(\mu,D)-\E^\delta}_{L^2(\Omega)}^2 \\ 
	&+ \frac{\alpha_\mu}{2} \int_\Omega (v_\mu^2+\zeta_\mu) |\nabla \mu|^2 \,dx + \frac{\alpha_D}{2} \int_{\Omega} (v_D^2+\zeta_D) |\nabla D|^2 \,dx \\
	&+ \beta_\mu \int_\Omega \left(\epsilon |\nabla v_\mu|^2 + \frac{(v_\mu-1)^2}{4\epsilon} \right) dx \\
	&+ \beta_D \int_\Omega \left(\epsilon |\nabla v_D|^2 +  \frac{(v_D-1)^2}{4\epsilon} \right) dx.
 \end{aligned}	
\end{equation}
It is well-known that, for all $\alpha,\beta > 0$ and $\zeta=o(\epsilon)$,
\begin{equation*}
	\alpha \int_{\Omega} (v^2+\zeta) |\nabla f|^2 + \beta \left(\epsilon |\nabla v|^2 + \frac{(v-1)^2}{4\epsilon} \right) dx \underset{\epsilon \to 0}{\longrightarrow} \alpha \int_{\Omega} |\nabla f|^2 \,dx + \beta \,H^{N-1}[S(f)]
\end{equation*}
in the sense of $\Gamma$-convergence in $L^1(\Omega)$ (formally, the latter has to be extended to an additional variable, see \cite{AmbTor92,RonSan01}). Since the data term is continuous in $L^1(\Omega)$ and $\Gamma$-convergence is stable under continuous perturbations, we thus get $\glim_{\epsilon \to 0} \F_\epsilon = \overline\F$ and therefore $L^1(\Omega)$-convergence of (a subsequence of) $\F_\epsilon$-minimizers (which can be shown to lie in a compact set \cite{AmbTor92}) to $\overline\F$-minimizers. 

The existence of a minimizer  of \eqref{eq:qpat_ms_at_funct} can be established using the direct method.
\begin{proposition}
\label{prop:existence_minimizers_at}
The functional $\F_\epsilon$ has at least one minimizer $(\hat \mu, \hat D, \hat{v}_\mu, \hat{v}_D)$ in $X_a^b(\Omega)^2 \cap W^{1,2}(\Omega)^2 \times X_0^1(\Omega)^2 \cap W^{1,2}(\Omega)^2$.
\end{proposition}
\begin{proof}
$F_\epsilon$ is coercive with respect to the semi-norm
\begin{equation*}
	|\mu|_{W^{1,2}(\Omega)} + |D|_{W^{1,2}(\Omega)} + \norm{v_\mu}_{W^{1,2}(\Omega)} + \norm{v_D}_{W^{1,2}(\Omega)},
\end{equation*}
which, combined with the restriction $(\mu,D) \in X_a^b(\Omega)^2$, shows that a minimizing sequence is bounded in $W^{1,2}(\Omega)^4$, so it contains a  subsequence that converges weakly to $(\hat \mu, \hat D, \hat{v}_\mu, \hat{v}_D)$ in $W^{1,2}(\Omega)^4$ (and thus strongly in $L^2(\Omega)^4$). Moreover, note that the spaces $X_a^b(\Omega)$ and $X_0^1(\Omega)$ are closed in $L^2(\Omega)$. The Proposition follows since the discrepancy term is continuous in $L^2(\Omega)^2$ and the regularization terms are weakly lower semi-continuous in $W^{1,2}(\Omega)^4$ (cf. \cite[Theorem 1.13]{Dac08} for the $\alpha$-terms). 
\end{proof}

\begin{remark}
\label{rem:piecewise_constant_limit}
If we vary $\alpha_\mu,\alpha_D$ with $\epsilon$ and take 
\begin{equation*}
	\alpha_\mu,\alpha_D \underset{\epsilon \to 0}{\longrightarrow} \infty
\end{equation*}
the minimizers of $\F_\epsilon$ converge to piecewise constant (in $\SBV$-sense, that is, with gradient vanishing almost everywhere) functions $(\hat \mu,\hat D)$  that minimize
\begin{equation*}
		\widetilde{\F}(\mu,D) = \frac{1}{2} \norm{\E(\mu,D)-\E^\delta}_{L^2(\Omega)}^2 + \beta_\mu \,H^{N-1}[S(\mu)] + \beta_D \,H^{N-1}[S(D)]
\end{equation*}
among all piecewise constant $\mu,D$. A proof (for the single variable case) can be found in \cite{AmbTor92}. This explains, in light of Section \ref{sec:recovery_pw_const}, why this type of regularization is useful for qPAT. 
\end{remark}

\subsection{Minimization}
\label{sec:mumford_shah_minimization} 
In this subsection, we proceed formally, i.e., without proving convergence, existence of minimizers of sub-problems and that the necessary derivatives and adjoints exist. For the minimization of \eqref{eq:qpat_ms_at_funct}, we suggest the following alternating directions approach:
\FloatBarrier
\begin{algorithm}
\label{alg:qpat_ms}
~ \\
\vspace{-\baselineskip}
\begin{enumerate}

	\item Choose parameters $a,b,\epsilon,\alpha_\mu,\alpha_D,\beta_\mu,\beta_D > 0$.
	\item Find an initial estimate $\hat K$ of the edge set $J_0(\mu) \cup J_0(D)$ using edge detection.
	\item Initialize with $\mu^0 = D^0 \equiv 1$ and $v^0_\mu = v^0_D = 1-1_{\hat K}$ .
	\item Iterate until convergence:

	\begin{itemize}
		\item[(i)] $(\mu^{n+\frac 1 2},D^{n+\frac 1 2})= \argmin_{\mu,D} \ \F_\epsilon(\mu,D,v^n_\mu,v^n_D)$
		\item[(ii)] $(v^{n+1}_\mu,v^{n+1}_D) = \argmin_{v_\mu,v_D} \ \F_\epsilon(\mu^{n+\frac 1 2},D^{n+\frac 1 2},v_\mu,v_D)$
	\end{itemize}

\end{enumerate}
\end{algorithm} 
\FloatBarrier
We now explain the steps in more detail.

\subsubsection*{Initialization using edge detection}
In our numerical simulations, it became clear that using Proposition \ref{prop:pw_const_prop_1} to obtain a reasonable estimate $\hat K$ of $J_0(\mu) \cup J_0(D)$ leads to faster convergence and better minimizers than simply taking $\hat K = \emptyset$. 

Given the noisy measurements $\E^\delta$, we have to estimate the discontinuities (or edges) of $\E^\delta$, $|\nabla \E^\delta|^2$ and $|\Delta \E^\delta|$. Since the jumps in these functions are of \emph{multiplicative} nature (proportional to the local value of ${u}$ or $\nabla {u} \cdot \nu$), it is advantageous to apply a logarithmic transformation prior to edge detection (to obtain constant contrast), see \cite{NaeSch14}. 

Due to noise, the data has to be smoothed prior to taking derivatives (turning discontinuities into areas with large gradients). We smooth by convolution with a Gaussian, since this approach has the advantage that differentiation and smoothing can be done in one step (by differentiating the low-pass filter instead of the function). We have

\begin{equation*}
\begin{aligned}
	\operatorname{G}_\sigma(x) &= (2\pi)^{-\frac{N}{2}} \sigma^{-N} \exp\left(\frac{-x^2}{2\sigma^2} \right) \\
	(\nabla \operatorname{G}_\sigma)(x) &= (2\pi)^{-\frac{N}{2}} \sigma^{-(N+2)} (-x) \exp\left(\frac{-x^2}{2\sigma^2} \right)  \\
	(\Delta \operatorname{G}_\sigma)(x) &= (2\pi)^{-\frac{N}{2}} \sigma^{-N+2} \left( \frac{|x|^2}{\sigma^2}-N \right) \exp\left(\frac{-x^2}{2\sigma^2} \right).
\end{aligned}
\end{equation*}

Similar to \cite{NaeSch14}, we proceed as follows:
\begin{enumerate}

	\item Detect edges $\hat K_0$ of $f_0:=\log \left( \E^\delta \ast \ 1_{B}\,G_{\sigma_0} \right)|_{\Omega \ominus B}$
	\item Detect edges $\hat K_1$ of $f_1:=\log \left( \left| \left( \E^\delta \ast \ 1_{B} \, \nabla G_{\sigma_1} \right)  \right|^2 \vee \gamma \right))|_{(\Omega \setminus \hat K_0) \ominus B}$
	\item Detect edges $\hat K_2$ of  $f_2:=\log \left( \left| \left( \E^\delta \ast \ 1_{B} \, \Delta G_{\sigma_2} \right) \right| \right)|_{(\Omega \setminus (\hat K_0 \cup \hat K_1)) \ominus B}$	
	\item Take $\hat K = \hat K_0 \cup \hat K_1 \cup \hat K_2$

\end{enumerate}

Here, $1_{B}$ acts as a cut-off function for the filters (e.g., using $B=B^p_\rho$, the $p$-norm ball of radius $\rho$). The operation $A \ominus B = \{ z \in \Omega \mid (B + z) \subset A \}$ denotes set erosion, which is performed to avoid multiple detection of edges (by ensuring that the cut-off filter does not intersect with already detected edges or the outside of the domain). Note, however, that for large $B$ this may lead to parts of edges (close to the domain boundary or already detected edges) not being detected. The scalar $\gamma$ is a minimal value enforced for $|\nabla \E^\delta|^2$ (to avoid creating singularities at zeros).  

The edge detection itself is performed by applying thresholds $\xi_0$,$\xi_1$,$\xi_2$ to the functions $|\nabla f_0|, |\nabla f_1|, |\nabla f_2|$ (taking the super-level sets as edge sets). Note that in contrast to \cite{NaeSch14}, no complete segmentation is necessary and cruder edge estimates suffice, that is, the detected edges don't have to be reduced to thin curves.

\subsubsection*{Iteration}
In step (i), to find, for fixed $v_\mu,v_D \in X_0^1(\Omega) \cap W^{1,2}(\Omega)$, 
\begin{equation*}
	\underset{\mu,D \in X_a^b(\Omega) \cap W^{1,2}(\Omega)}{\argmin} \ \F_\epsilon(\mu,D)
\end{equation*}
we use Gauss-Newton-minimization. That is, in every iteration of an inner loop, we linearly approximate $\E(\mu_k+s_\mu,D_k+s_D) \approx \E(\mu_k,D_k) + \E'(\mu_k,D_k)(s_\mu,s_D)$ and take update steps
\begin{equation}
\label{eq:qpat_ms_gauss_newton}
 \begin{aligned}
	s &= \underset{s_\mu,s_D \in W^{1,2}(\Omega)}{\argmin} \ \frac{1}{2} \norm{\E(\mu_k,D_k)+\E'(\mu_k,D_k)(s_\mu,s_D)-\E^\delta}_{L^2(\Omega)}^2  \\ 
	&+ \frac{\alpha_\mu}{2} \int_{\Omega} (v_\mu^2+\zeta_\mu) |\nabla (\mu_k+s_\mu)|^2 \,dx + \frac{\alpha_D}{2} \int_{\Omega} (v_D^2+\zeta_D) |\nabla (D_k+s_D)|^2 \,dx \\
 \end{aligned}		
\end{equation}
which we then project into $X_a^b(\Omega)^2$, so we update with
\begin{equation*}
	(\mu_{k+1},D_{k+1}) = (\mu_k,D_k) + (s \wedge a) \vee b.
\end{equation*}	
A straightforward calculation shows that a minimizer $s=(s_\mu,s_D)$ of \eqref{eq:qpat_ms_gauss_newton} satisfies the weak form of
\begin{equation}
\label{eq:qpat_ms_gauss_newton_el}
\begin{aligned}
	\E'&(\mu_k,D_k)^*\E'(\mu_k,D_k)(s_\mu,s_D) + \left(\alpha_\mu L_{(v_\mu^2 + \zeta_\mu)} \,s_\mu,\ \alpha_D L_{(v_D^2 + \zeta_D)} \,s_D \right) \\
	&= -\E'(\mu_k,D_k)^*[ \E(\mu_k,D_k)-\E^\delta] - \left(\alpha_\mu L_{(v_\mu^2 + \zeta_\mu)} \,\mu_k,\ \alpha_D L_{(v_D^2 + \zeta_D)} \,D_k \right) \\
	&=: (R_{\mu_k},R_{D_k})
\end{aligned}	
\end{equation}
with homogeneous Neumann boundary conditions for $s_\mu,s_D$ and
\begin{equation*}
	L_\sigma\colon W^{1,2}(\Omega) \to W^{-1,2}(\Omega), \ x \mapsto -\nabla \cdot ( \sigma \nabla x ).
\end{equation*}
The linear operator $\E'(\mu,D)$ and its (formal) adjoint $\E'(\mu,D)^*$ are given by  
\begin{equation}
\label{eq:qpat_derivatives_adjoints}
\begin{aligned}
	\E'(\mu,D)(s_\mu,s_D) &= s_\mu {u} + \mu L_{\mu,D}^{-1}(-s_\mu {u}) + \mu L_{\mu,D}^{-1}[\nabla \cdot ( s_D \nabla {u} )] \\
	\E'(\mu,D)^*\,t &= \left( t {u} - {u} L_{\mu,D}^{-1}(\mu t), -\nabla {u} \cdot \nabla L_{\mu,D}^{-1}(\mu t) \right)
\end{aligned}
\end{equation}
where ${u}={u}(\mu,D)$ and
\begin{equation*}
	L_{\mu,D}\colon W_0^{1,2}(\Omega) \to W^{-1,2}(\Omega), \ x \mapsto -\nabla \cdot ( D \nabla x ) + \mu x.
\end{equation*}
By introducing auxiliary variables $y_1,y_2,y_3$  (and taking $\mu=\mu_k, D=D_k$), the equations \eqref{eq:qpat_ms_gauss_newton_el},\eqref{eq:qpat_derivatives_adjoints} can be re-written as the system of (weak) PDE
\begin{equation}
\label{eq:qpat_ms_system}
\begin{aligned}
	 (s_\mu {u} + \mu y_1 + \mu y_2) {u} - {u} y_3 + \alpha_\mu L_{(v_\mu^2 + \zeta_\mu)} \,s_\mu = R_\mu, \quad \frac{\partial s_\mu|_{\partial\Omega}}{\partial \nu}&=0 \\
	 -\nabla {u} \cdot \nabla y_3  + \alpha_D L_{(v_D^2 + \zeta_D)} \,s_D  =  R_D, \quad \frac{\partial s_D|_{\partial\Omega}}{\partial \nu}&=0 \\
	L_{\mu,D} \, y_3 = \mu (s_\mu {u} + \mu y_1 + \mu y_2), \quad y_3|_{\partial\Omega}&=0 \\
	L_{\mu,D} \, y_2 = \nabla \cdot ( s_D \nabla {u} ), \quad y_2|_{\partial\Omega}&=0 \\
	L_{\mu,D} \, y_1 = -s_\mu {u}, \quad y_1|_{\partial\Omega}&=0 
\end{aligned}
\end{equation}
which is amenable to discretization as a sparse matrix using, e.g., the \emph{finite element method (FEM)}.

In step (ii) of the outer loop we have to find, for fixed $\mu,D \in X_a^b(\Omega) \cap W^{1,2}(\Omega)$,
\begin{equation*}
	\underset{v_\mu,v_D \in X_0^1(\Omega) \cap W^{1,2}(\Omega)}{\argmin} \ \F_\epsilon(v_\mu,v_D).
\end{equation*}

The minimizer $(v_\mu,v_D)$ satisfies the linear equations 
\begin{equation} 
\label{eq:qpat_ms_edges_el}
\begin{aligned}
 \left(-2 \beta \epsilon \Delta +  (\alpha_\mu |\nabla \mu|^2  + \frac{\beta_\mu}{2 \epsilon}) \Id \right) \, v_\mu &= \frac{\beta}{2 \epsilon}, \quad \frac{\partial v_\mu|_{\partial\Omega}}{\partial \nu}=0 \\
  \left(-2 \beta \epsilon \Delta +  (\alpha_D |\nabla D|^2  + \frac{\beta_D}{2 \epsilon}) \Id \right) \, v_D &= \frac{\beta}{2 \epsilon}, \quad \frac{\partial v_D|_{\partial\Omega}}{\partial \nu}=0 \\
\end{aligned}
\end{equation}
which can be solved separately and implemented numerically using FEM.

\section{Implementation and numerical results}
\label{sec:numerical_results}

In this section, the proposed algorithm is tested on two sets of simulated two-dimensional data (with different parameter range and level of detail). We also vary the noise on the data since the reconstruction quality strongly depends on the noise level. 

The data (see Figures \ref{fig:setup_circles} and \ref{fig:setup_rectangles}) was generated in the diffusion model \eqref{eq:diff_eq_dirichlet} using self-written (linear-basis) finite element code in \emph{MATLAB}. For both examples, we took $\Omega=[0,5]^2$ and used a uniform boundary condition $g \equiv 1$. The simulated data were generated on a $(400\times400)$-grid and then down-sampled (by averaging) to $(200\times200)$ to avoid inverse crime.  After that, Gaussian noise with different intensities (standard deviations of $0.5\%$ and $10\%$ of the average signal value $\frac{1}{|\Omega|} \int_\Omega \E(x)\,dx$) was added to the data.

\FloatBarrier

\begin{figure}[ht]
	\centering
		\subfloat[$\mu$]{ \includegraphics[width=0.4\textwidth]{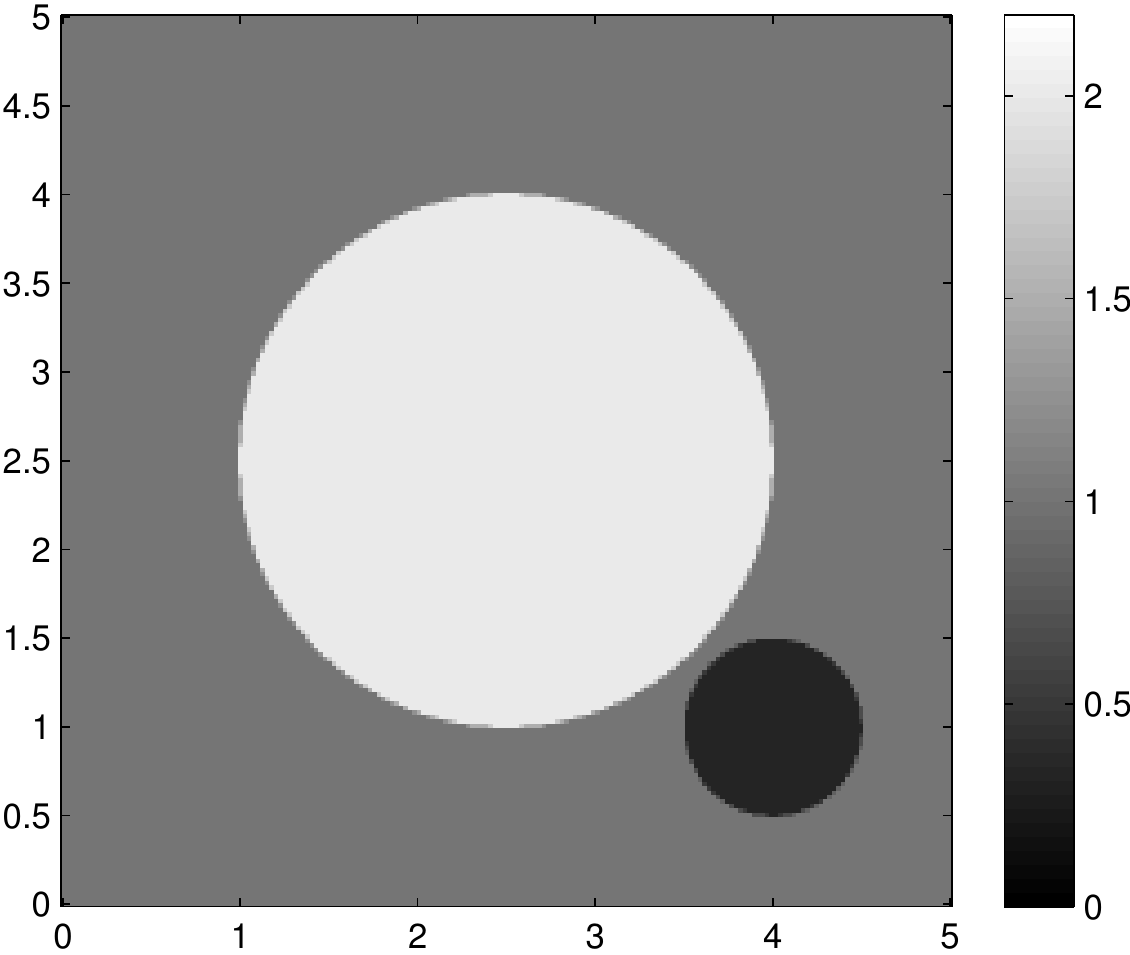} }  \hspace{\baselineskip}
  		\subfloat[$D$]{ \includegraphics[width=0.4\textwidth]{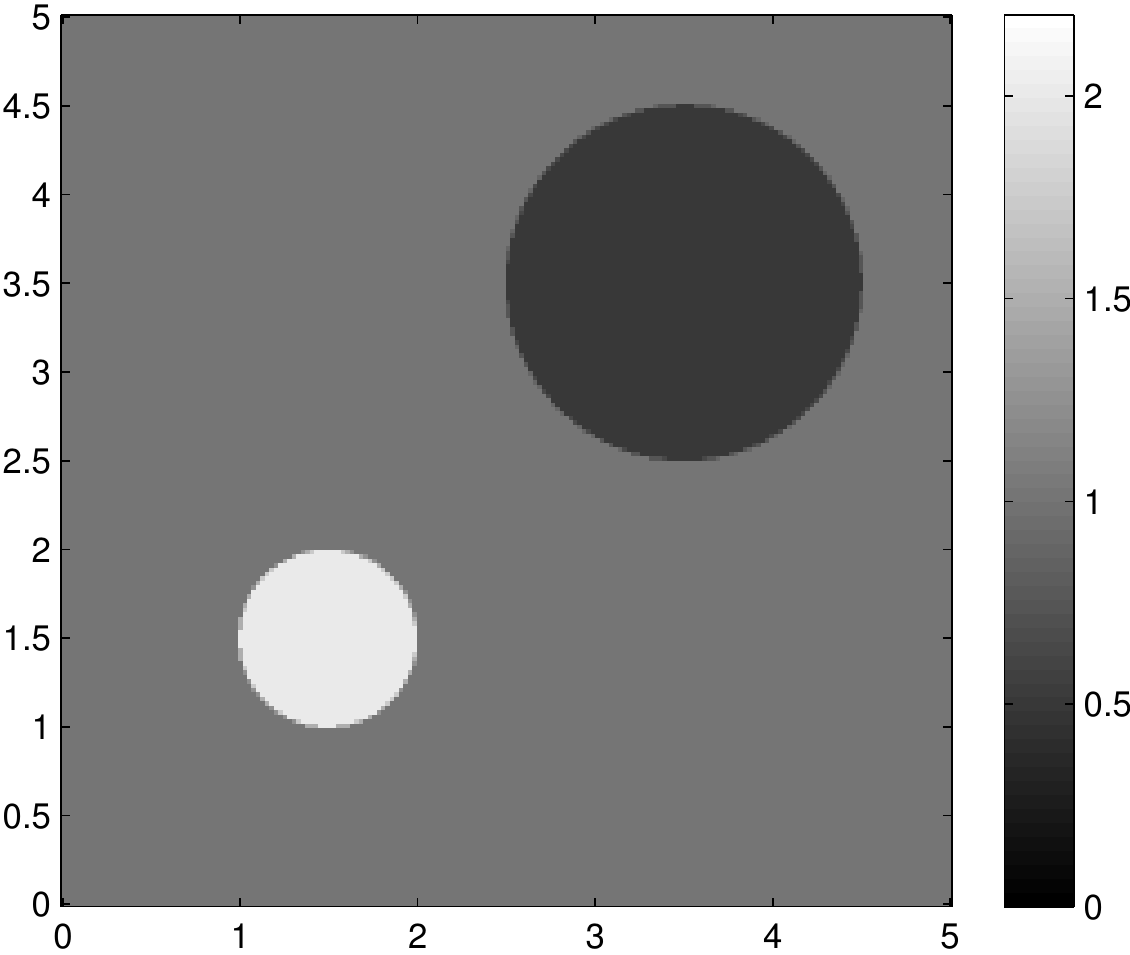} } \\ 
		\subfloat[$\E(\mu,D)$]{ \includegraphics[width=0.4\textwidth]{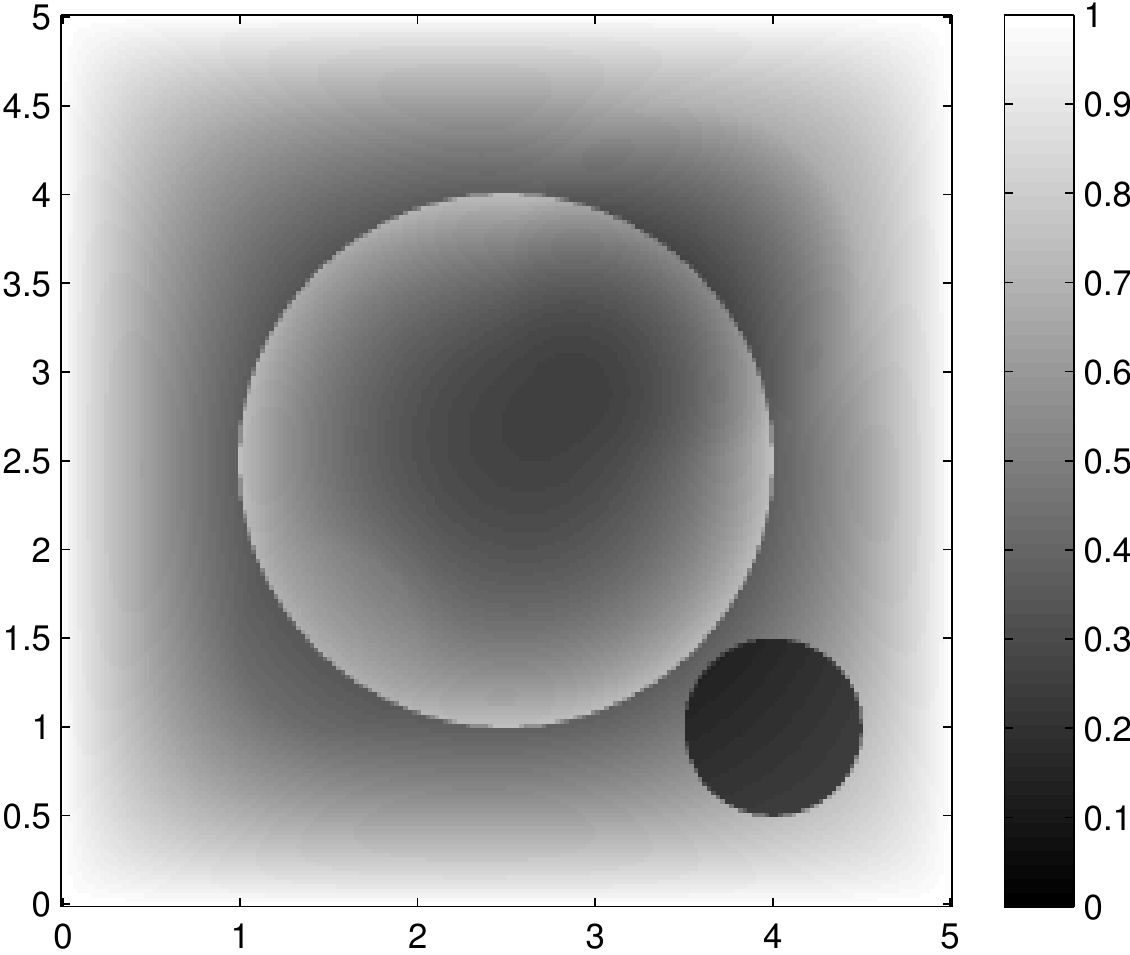} }  \hspace{\baselineskip}
  		\subfloat[$\E(\mu,D)^\delta$ ($10\%$ noise)]{ \includegraphics[width=0.4\textwidth]{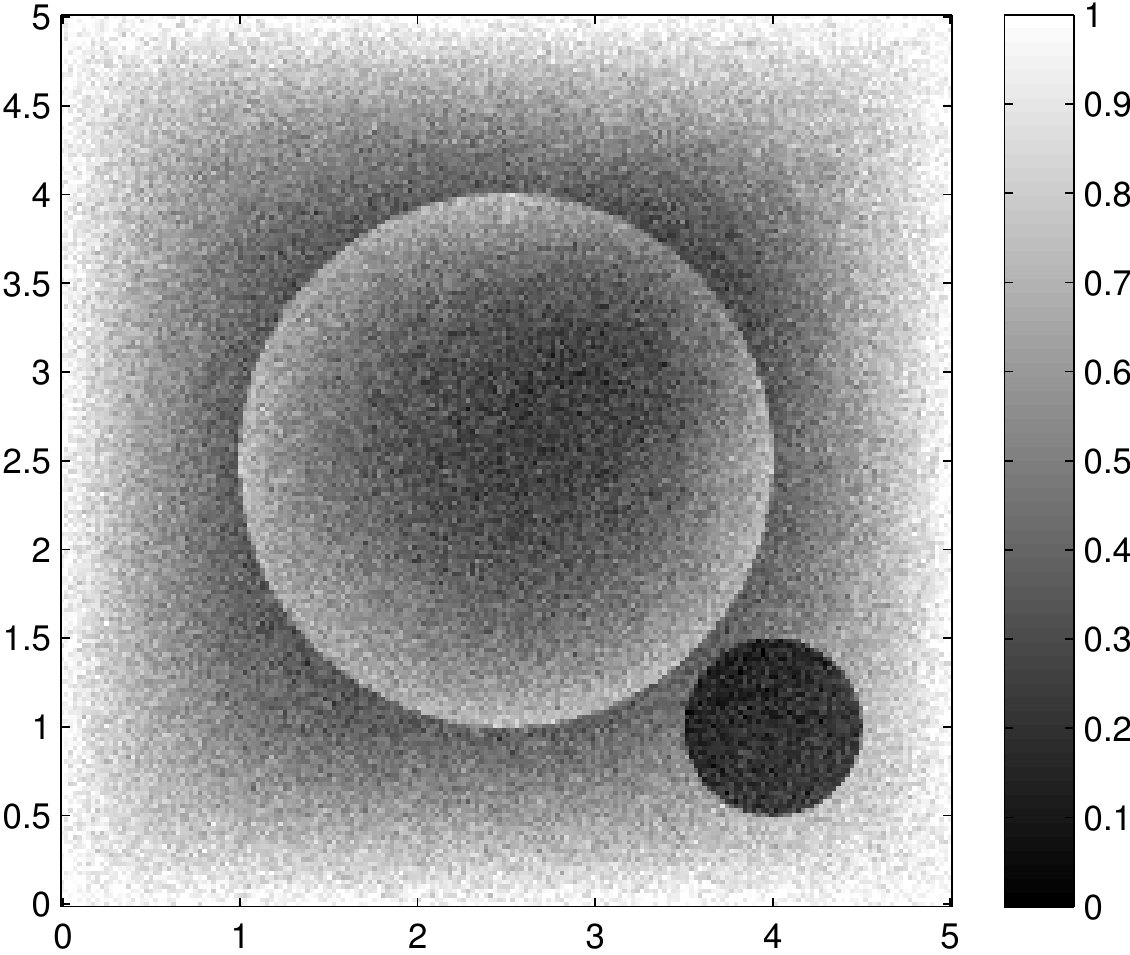} } \\
	\caption[Reconstruction results]
	{ Example A. Parameters $\mu,D$ and FEM-generated data $\E(\mu,D)$.}
	\label{fig:setup_circles}	
\end{figure}

\begin{figure}[ht]
	\centering
		\subfloat[$\mu$]{ \includegraphics[width=0.4\textwidth]{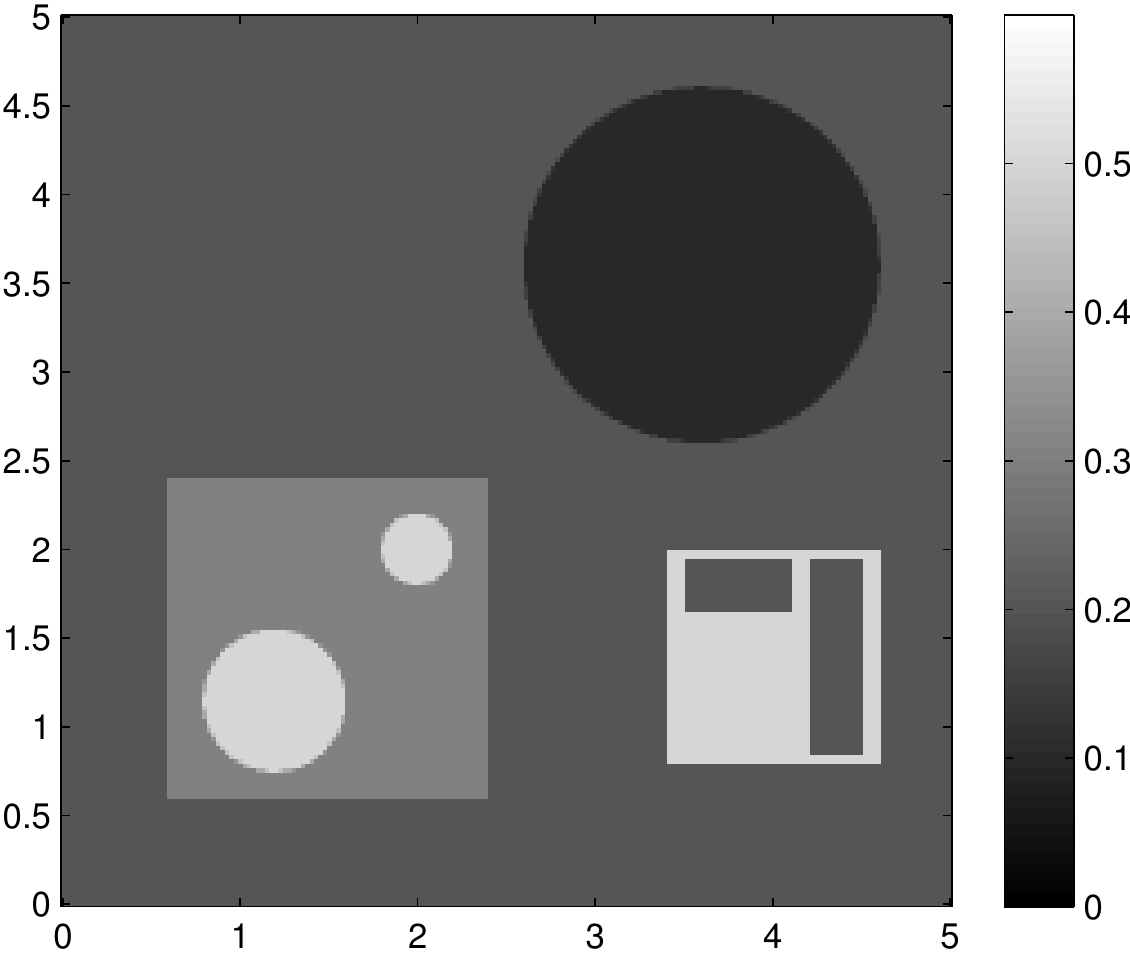} }  \hspace{\baselineskip}
  		\subfloat[$D$]{ \includegraphics[width=0.4\textwidth]{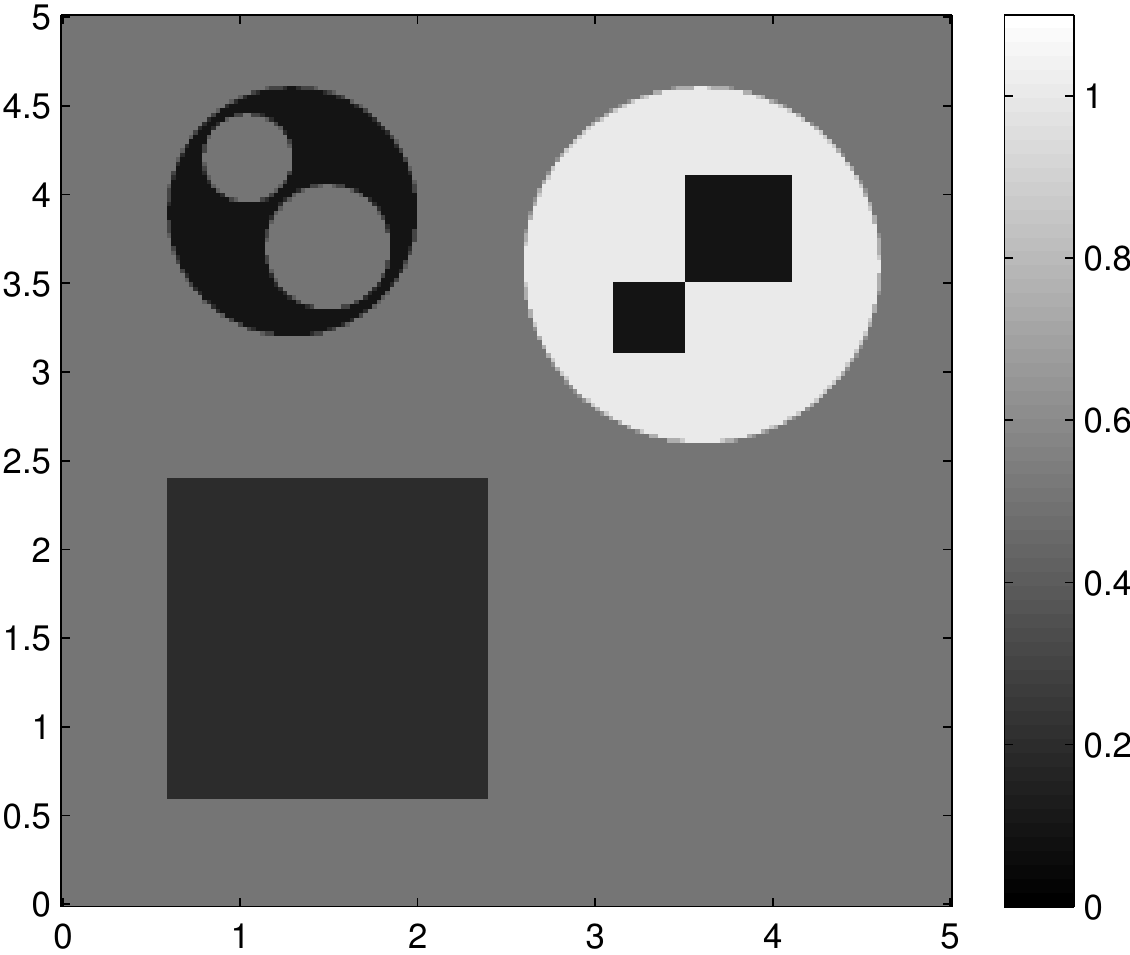} } \\  
		\subfloat[$\E(\mu,D)$]{ \includegraphics[width=0.4\textwidth]{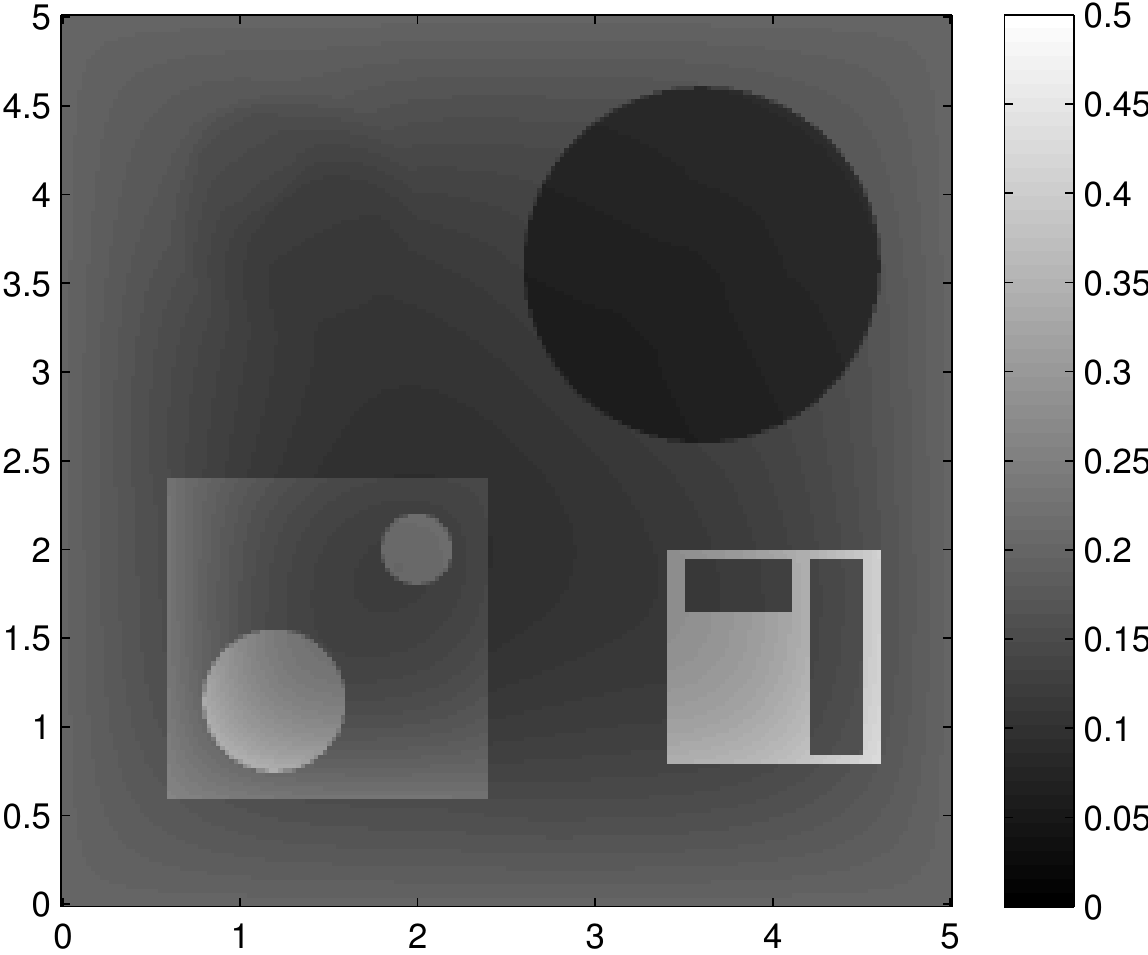} }  \hspace{\baselineskip}
  		\subfloat[$\E(\mu,D)^\delta$ ($10\%$ noise)]{ \includegraphics[width=0.4\textwidth]{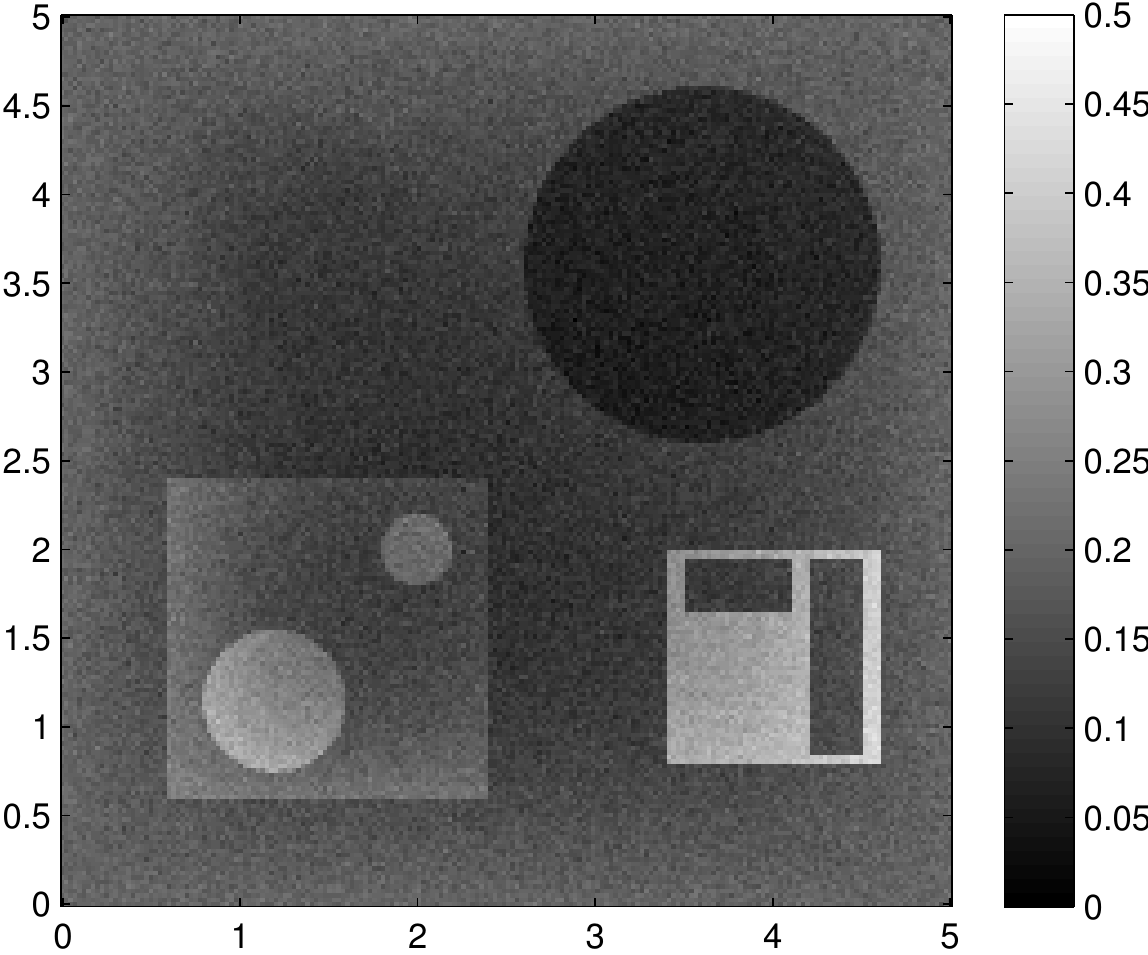} } \\
	\caption[Reconstruction results]
	{ Example B. Parameters $\mu,D$ and FEM-generated data $\E(\mu,D)$.}	
	\label{fig:setup_rectangles}	
\end{figure}

The edge detector described in Section \ref{sec:mumford_shah_minimization} was implemented by finite difference approximations of $|\nabla f_0|, |\nabla f_1|, |\nabla f_2|$  using central differences inside and one-sided differences near the boundary of the domain. The functions $f_1,f_2,f_3$ were calculated by convolution filtering with the \emph{MATLAB} function \emph{imfilter}. For all examples, we used a square cutoff function, i.e., $B=B^\infty_\rho$ with $\rho=2 \lceil{\sigma_k}\rceil$ (where $k=0,1,2$).

The edge detector is used to detect jumps in the derivatives of the data $\E^\delta$ up to second order (to obtain an initial estimate of the parameter jump set $J_0(\mu) \cup J_0(D)$). Since this process is highly sensitive with respect to noise, we varied the edge detection procedure subject to the amount of noise in the data. In the noise-free examples, we estimated the jumps of all three functions $f_0,f_1,f_3$, that is, jumps of derivatives of $\E^\delta$ up to second order. We restricted the jump estimation to $f_0,f_1$ for the low-noise examples (i.e., jumps of derivatives up to first order) and $f_0$ in the high-noise examples (only jumps in the data $\E^\delta$ itself).

To obtain the parameters $\mu,D$ given an initial estimate of their combined edge set $J_0(\mu) \cup J_0(D)$, we used the Ambrosio-Tortorelli approximation \eqref{eq:qpat_ms_at_funct} of the Mumford-Shah functional introduced in Section \ref{sec:mumford_shah_parameter_detection}. To minimize the functional, we used Algorithm \ref{alg:qpat_ms}, iterating, for all examples, the outer alternating-directions loop until the functional changes by less than $0.01\%$ and the inner Gauss-Newton loop until the functional value changes less than $1\%$. We directly implemented the systems \eqref{eq:qpat_ms_system} and \eqref{eq:qpat_ms_edges_el} using (self-written) linear-basis finite element code. The implementation is not fully conforming, that is, where necessary we used conversions between piecewise linear and piecewise constant functions for simplicity. The discrete systems were solved using the standard \emph{MATLAB} sparse equation solver \emph{mldivide}.

For all examples, we chose $a=0.01, b=3$ and $\epsilon=0.01$. The other reconstruction parameters used (which were selected by hand) are listed in Tables \ref{tab:results_parameters_iteration} and \ref{tab:results_parameters_edge_detection}. 

\begin{table}[ht]
	\centering
  		  \begin{tabular}[b]{ccccccc}	 		  	  
              					 & $\boldsymbol{\alpha_\mu}$ & $\boldsymbol{\alpha_D}$ & $\boldsymbol{\beta_\mu}$ &  $\boldsymbol{\beta_D}$ 
              					 & $\boldsymbol{\zeta_\mu}$ & $\boldsymbol{\zeta_D}$  \\
			  \toprule
		   	  \textbf{Ex. A} ($0 \%$ noise) & $10^{-2}$ & $10^{-4}$ & $10^{-6}$ & $10^{-8}$ & $10^{-6}$ & $10^{-4}$  \\
		   	  \midrule		   	    
		   	  \textbf{Ex. A} ($0.1 \%$ noise) & $10^{-2}$ & $10^{-5}$ & $10^{-6}$ & $10^{-8}$ & $10^{-6}$ & $10^{-3}$   \\
		   	  \midrule		   	  		   	  
		   	  \textbf{Ex. A} ($10 \%$ noise) & $1$ & $5 \cdot 10^{-3}$ & $10^{-5}$ & $5 \cdot 10^{-6}$ & $10^{-5}$ & $10^{-3}$  \\
		   	  \midrule		   	  
		   	  \textbf{Ex. B} ($0 \%$ noise) & $10^{-2}$ & $10^{-4}$ & $10^{-6}$ & $10^{-8}$ & $10^{-6}$ & $10^{-5}$  \\  
		   	  \midrule		   	  
		   	  \textbf{Ex. B} ($0.1 \%$ noise) & $10^{-2}$ & $10^{-5}$ & $10^{-6}$ & $10^{-8}$ & $10^{-6}$ & $10^{-3}$  \\
		   	  \midrule
		   	  \textbf{Ex. B} ($10 \%$ noise) & $1$ & $10^{-3}$ & $10^{-5}$ & $10^{-7}$ & $10^{-5}$ & $10^{-3}$ \\ 
		   	  \bottomrule 
  	      \end{tabular} 	  
	\caption[Reconstruction parameters]
	{ Parameters used for Figures \ref{fig:results_circles} and \ref{fig:results_rectangles} (for the functional \eqref{eq:qpat_ms_at_funct}). }
	\label{tab:results_parameters_iteration}	
\end{table} 

\begin{table}[ht]
\vspace{\baselineskip}
	\centering
  		  \begin{tabular}[b]{cccccccc}	 		  	  
              					 & $\boldsymbol{\sigma_0}$ & $\boldsymbol{\sigma_1}$ & $\boldsymbol{\sigma_2}$ & $\boldsymbol{\xi_0}$ & $\boldsymbol{\xi_1}$ & $\boldsymbol{\xi_2}$ & $\boldsymbol{\gamma}$ \\
			  \toprule
		   	  \textbf{Ex. A} ($0 \%$ noise) & $0.5$ & $0.5$ & $0.5$ & $0.1$ & $0.1$ & $0.1$ & $10^{-4}$\\
		   	  \midrule		   	    
		   	  \textbf{Ex. A} ($0.1 \%$ noise) & $0.5$ & $2$ & & $0.05$ & $0.06$ & & $10^{-5}$ \\
		   	  \midrule		   	  		   	  
		   	  \textbf{Ex. A} ($10 \%$ noise)  & $1.5$ & & & $0.05$ & & & \\
		   	  \midrule		   	  
		   	  \textbf{Ex. B} ($0 \%$ noise)  & $0.5$ & $0.5$ & $0.5$ & $0.1$ & $0.1$ & $0.1$ & $10^{-6}$  \\
		   	  \midrule		   	  
		   	  \textbf{Ex. B} ($0.1 \%$ noise) & $0.5$ & $2.4$ & & $0.05$ & $0.06$ & &  $10^{-7}$ \\
		   	  \midrule
		   	  \textbf{Ex. B} ($10 \%$ noise) & $1.5$ & & & $0.03$ & & &  \\ 
		   	  \bottomrule 
  	      \end{tabular} 	  
	\caption[Reconstruction parameters]
	{ Parameters used for Figures \ref{fig:results_circles} and \ref{fig:results_rectangles} (for edge detection). }
	\label{tab:results_parameters_edge_detection}	
\end{table} 

Reconstruction results and error profiles at different noise levels can be seen in Figures \ref{fig:results_circles} and \ref{fig:results_rectangles}. In both examples, the noise-free reconstructions are very accurate and contain mostly smoothing error. In the low-noise reconstructions, due to the fact that more regularization is necessary, some of the parameter variation is underestimated. In the high-noise examples, most detail in $D$ is lost since a lot of regularization is required to get reasonable results. The fine detail in $\mu$ can, however, still be recovered very accurately in both examples. 

\begin{figure}[ht]
	\centering
		\includegraphics[width=0.29\textwidth]{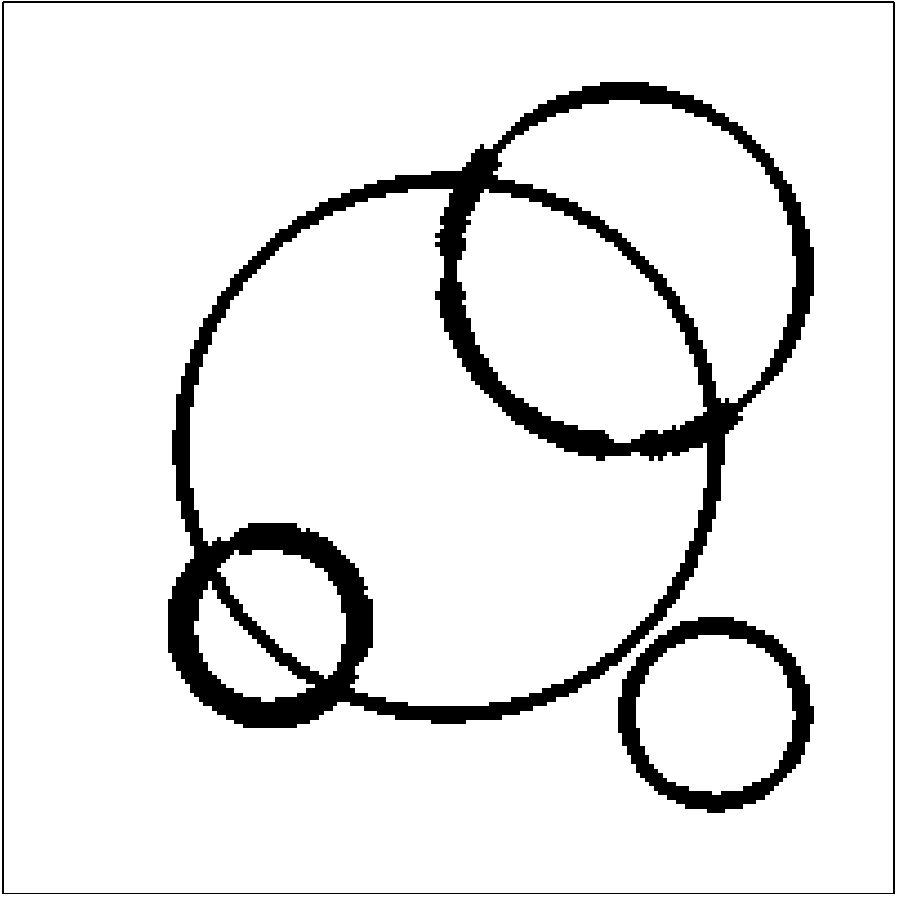}  \hspace{\baselineskip}
		\includegraphics[width=0.29\textwidth]{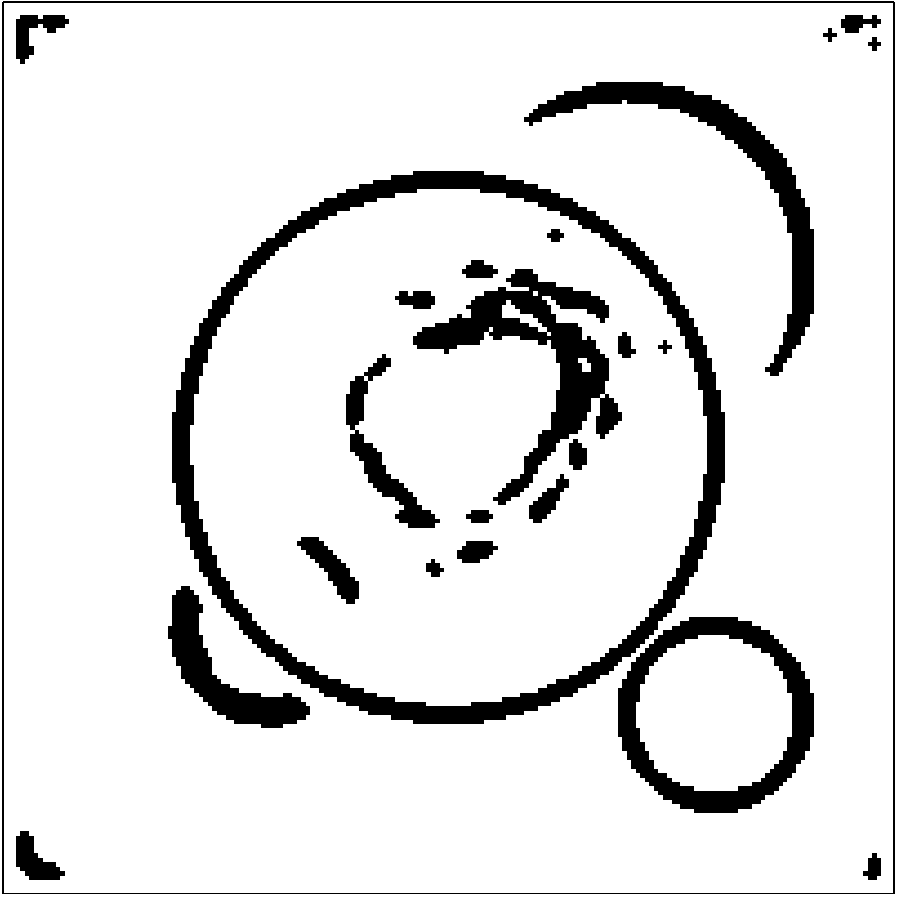}  \hspace{\baselineskip}		
		\includegraphics[width=0.29\textwidth]{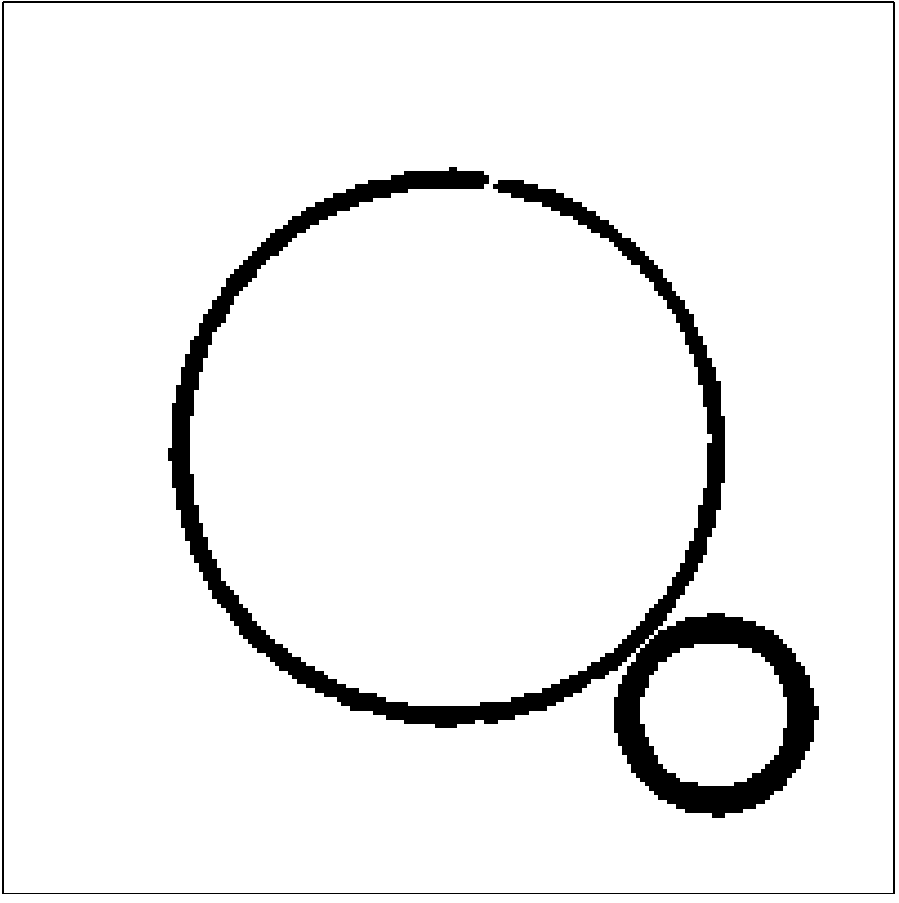}  \\ \vspace{\baselineskip}
		\includegraphics[width=0.29\textwidth]{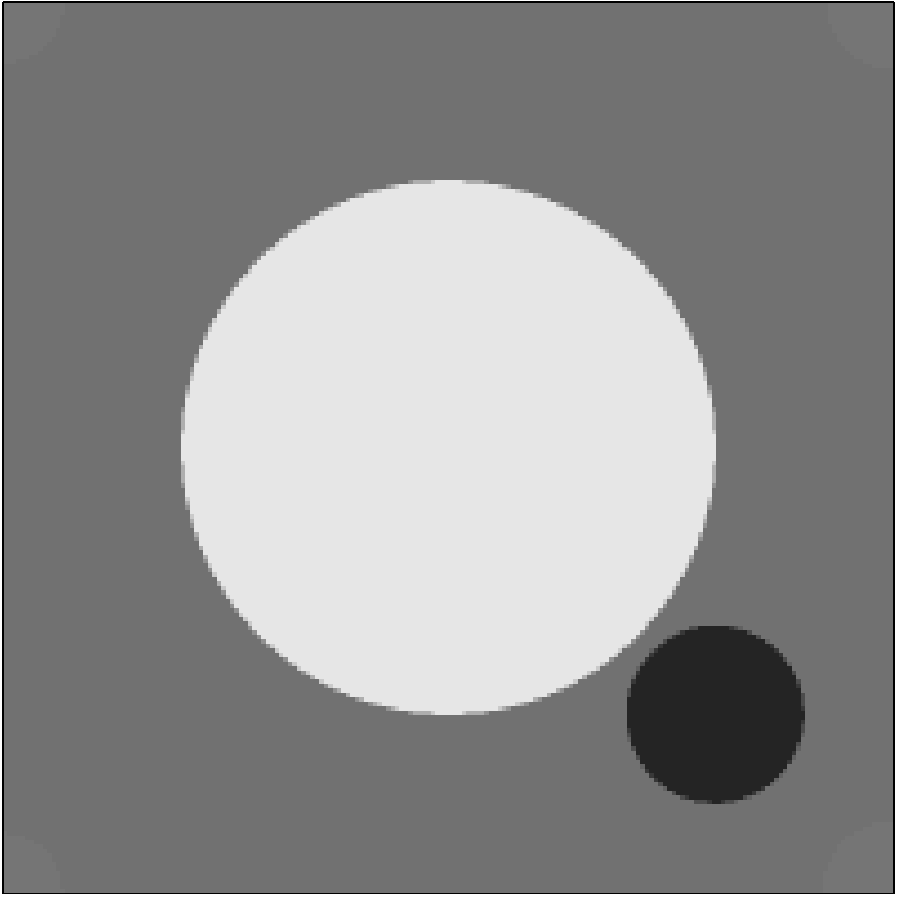}  \hspace{\baselineskip}
		\includegraphics[width=0.29\textwidth]{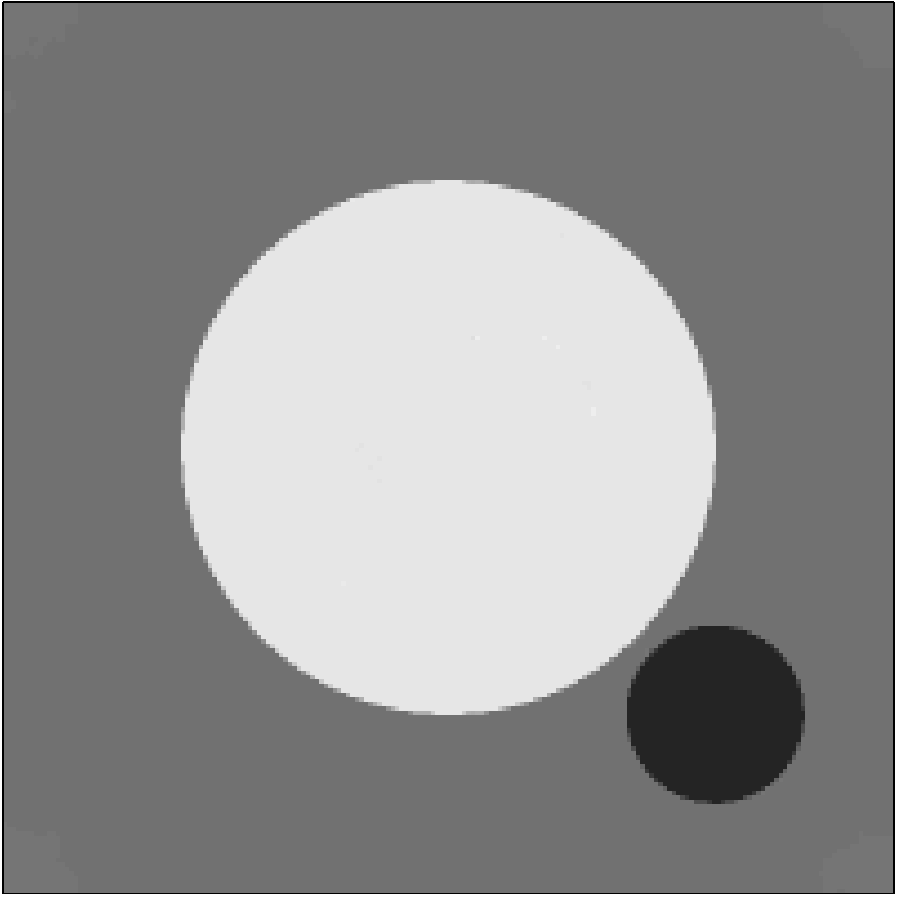}  \hspace{\baselineskip}		
		\includegraphics[width=0.29\textwidth]{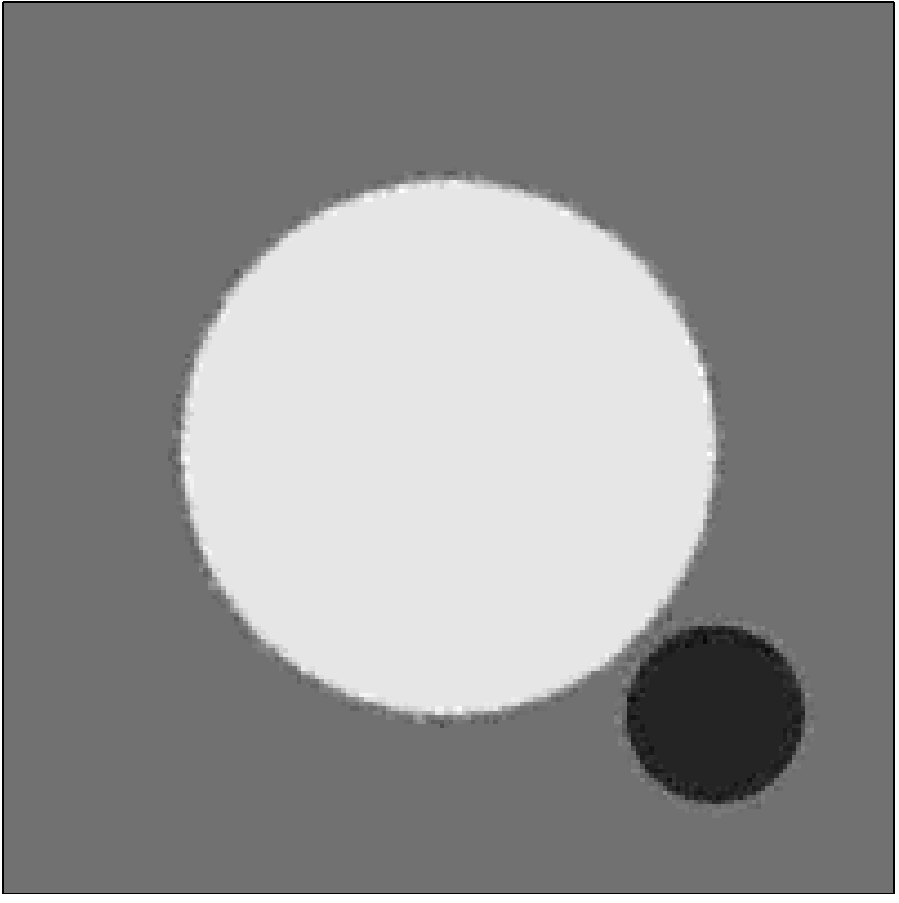}  \\	\vspace{\baselineskip}
		\includegraphics[width=0.29\textwidth]{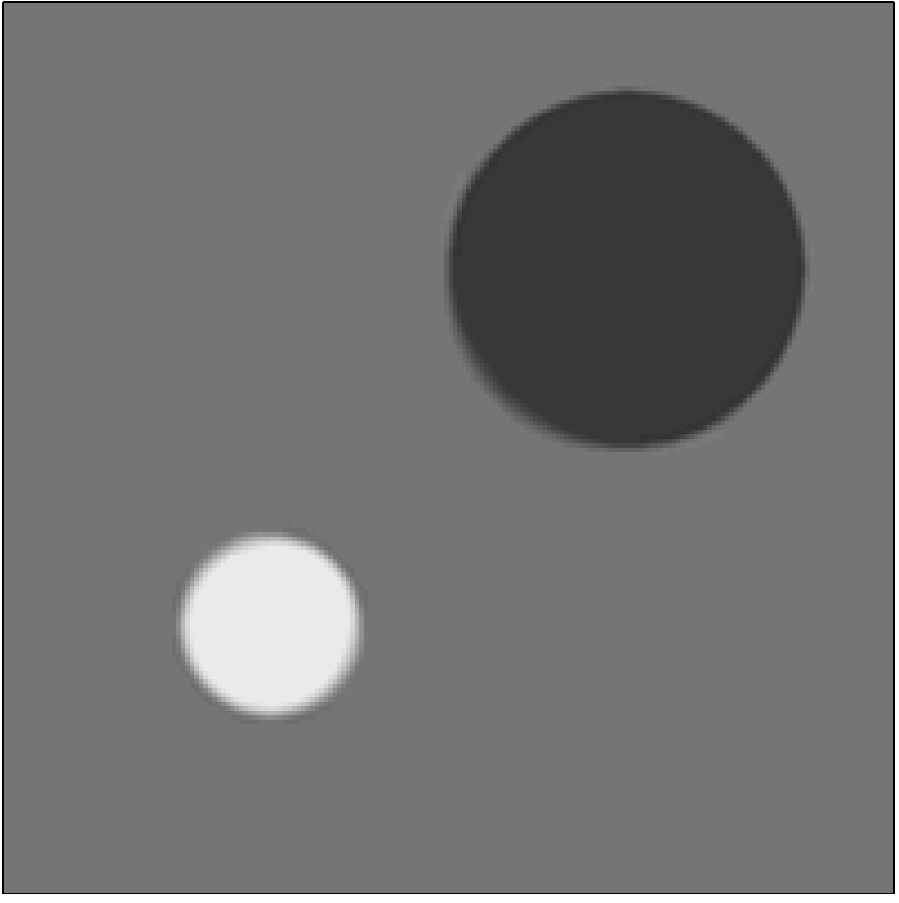}  \hspace{\baselineskip}
		\includegraphics[width=0.29\textwidth]{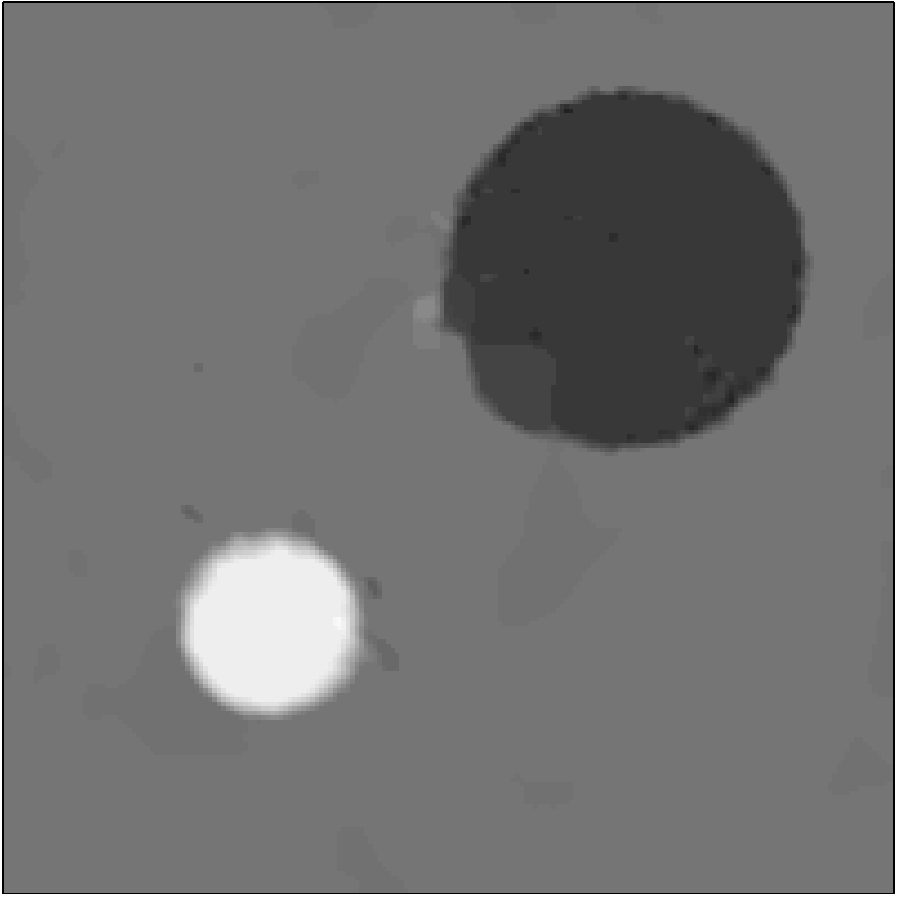}  \hspace{\baselineskip}		
		\includegraphics[width=0.29\textwidth]{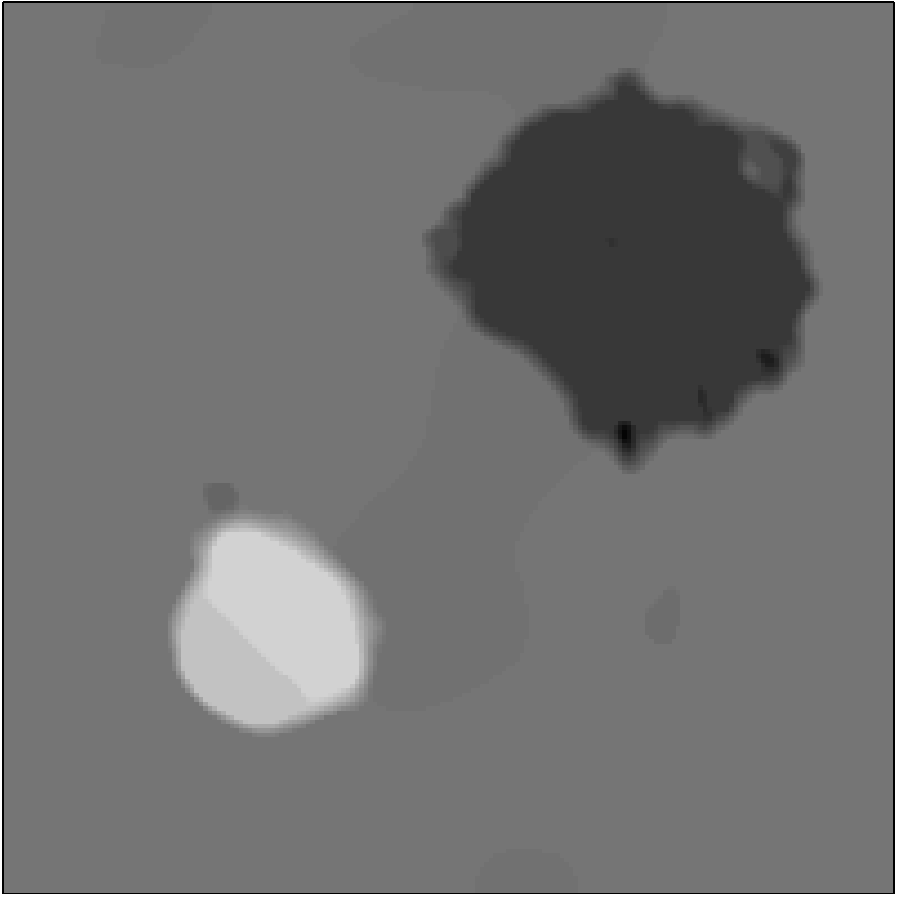}  \\ \vspace{\baselineskip}
		\includegraphics[width=0.29\textwidth]{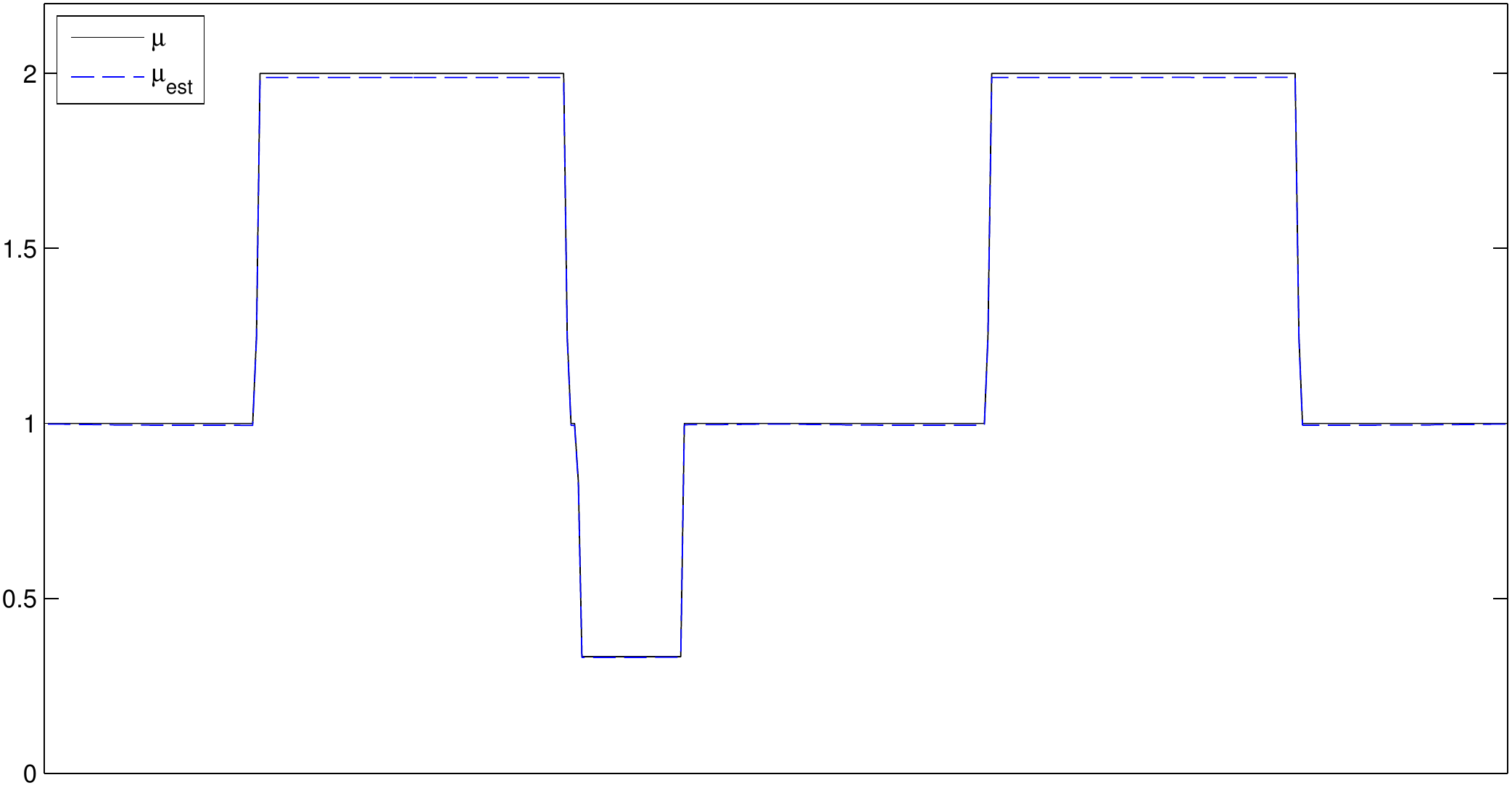}  \hspace{\baselineskip}
		\includegraphics[width=0.29\textwidth]{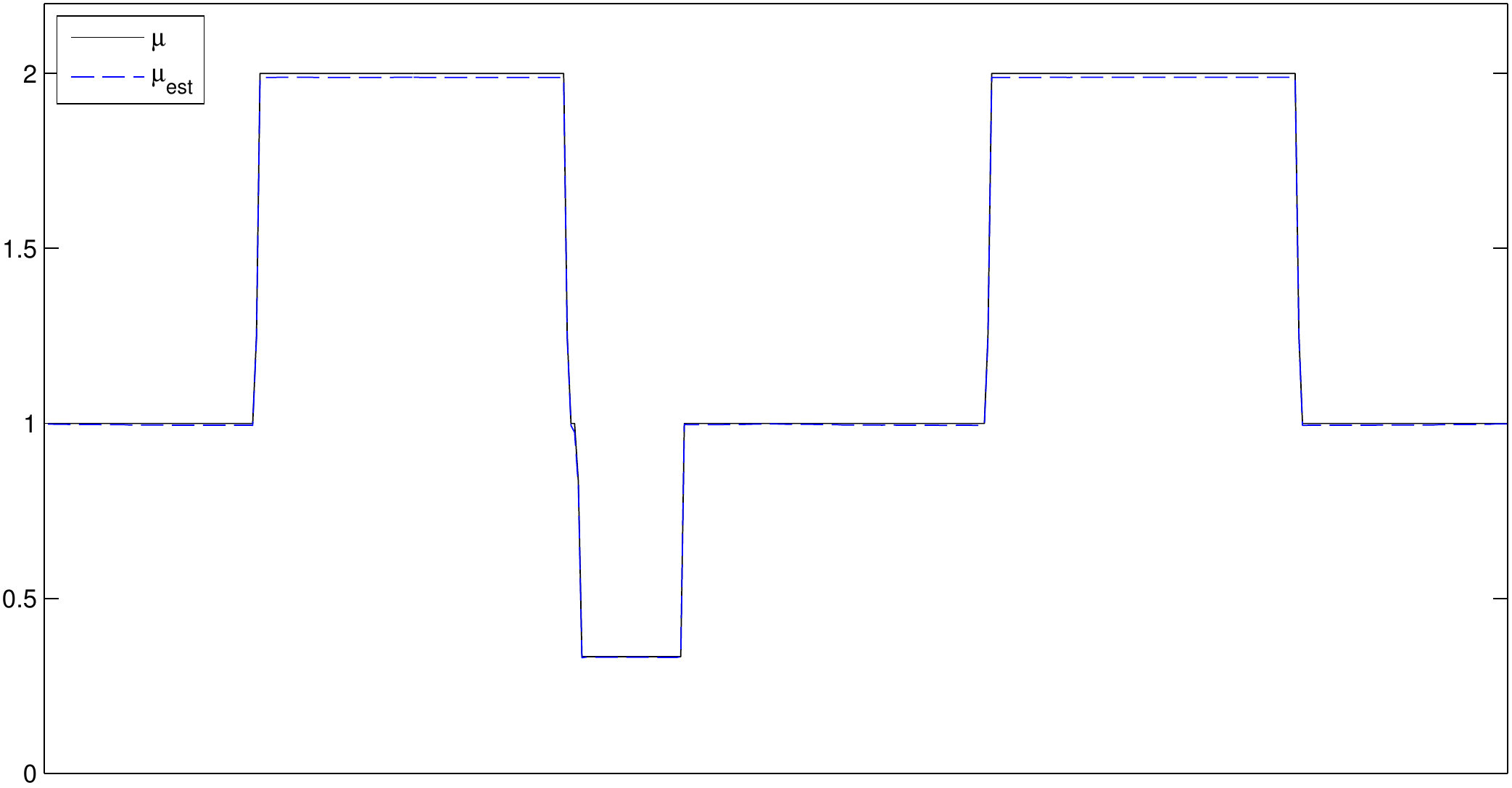}  \hspace{\baselineskip}		
		\includegraphics[width=0.29\textwidth]{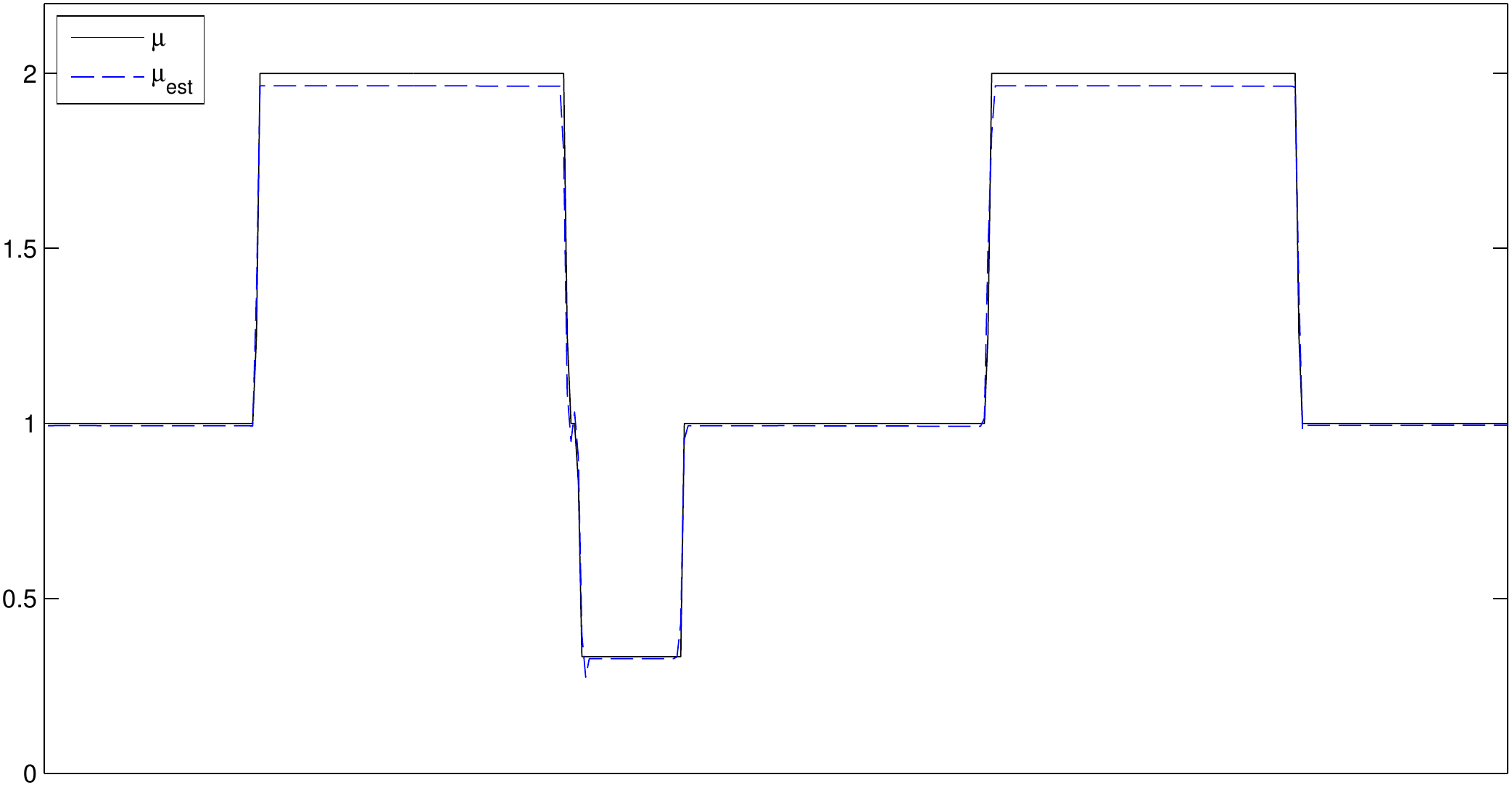}  \\	\vspace{\baselineskip}
		\includegraphics[width=0.29\textwidth]{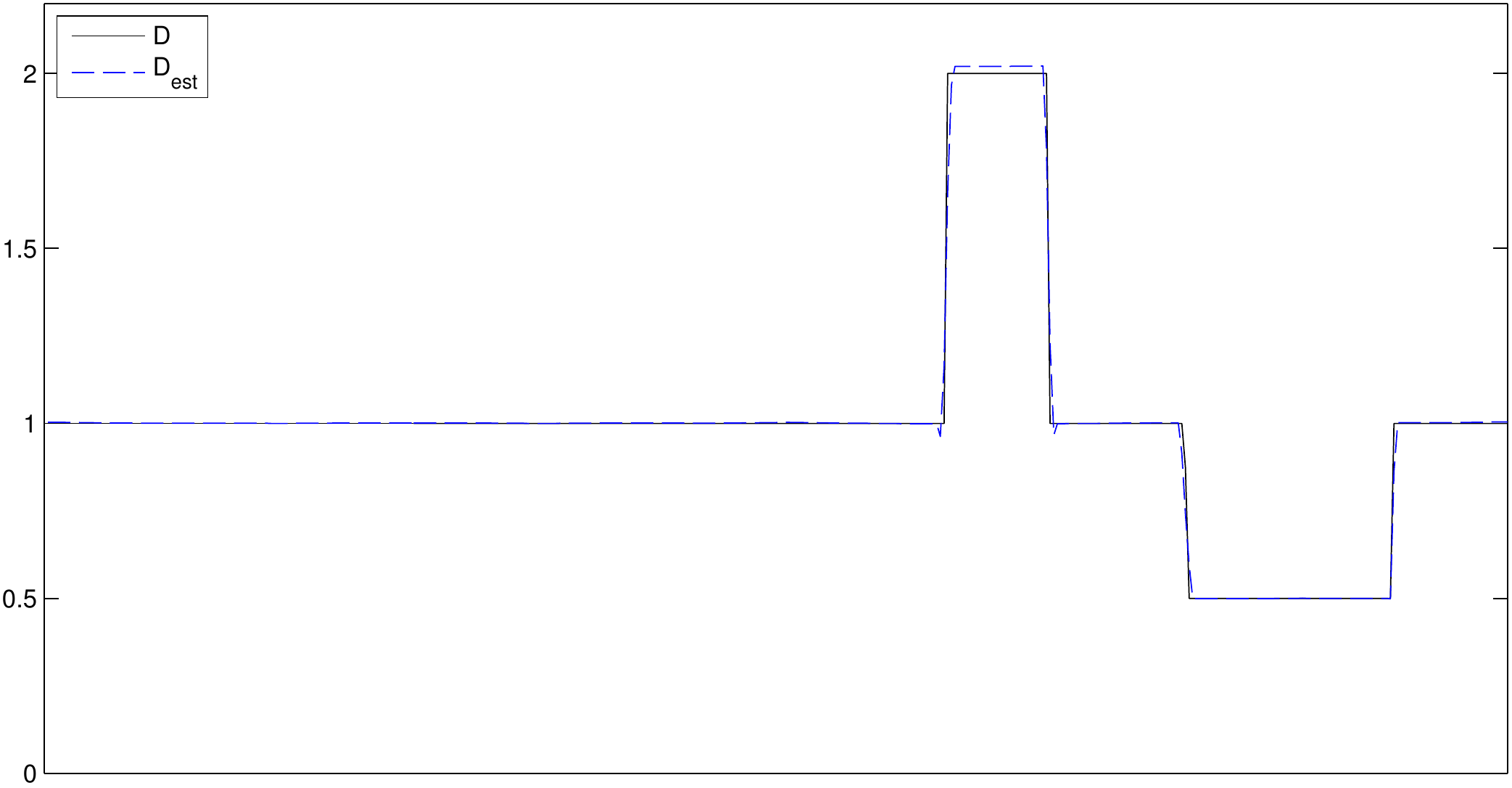}  \hspace{\baselineskip}
		\includegraphics[width=0.29\textwidth]{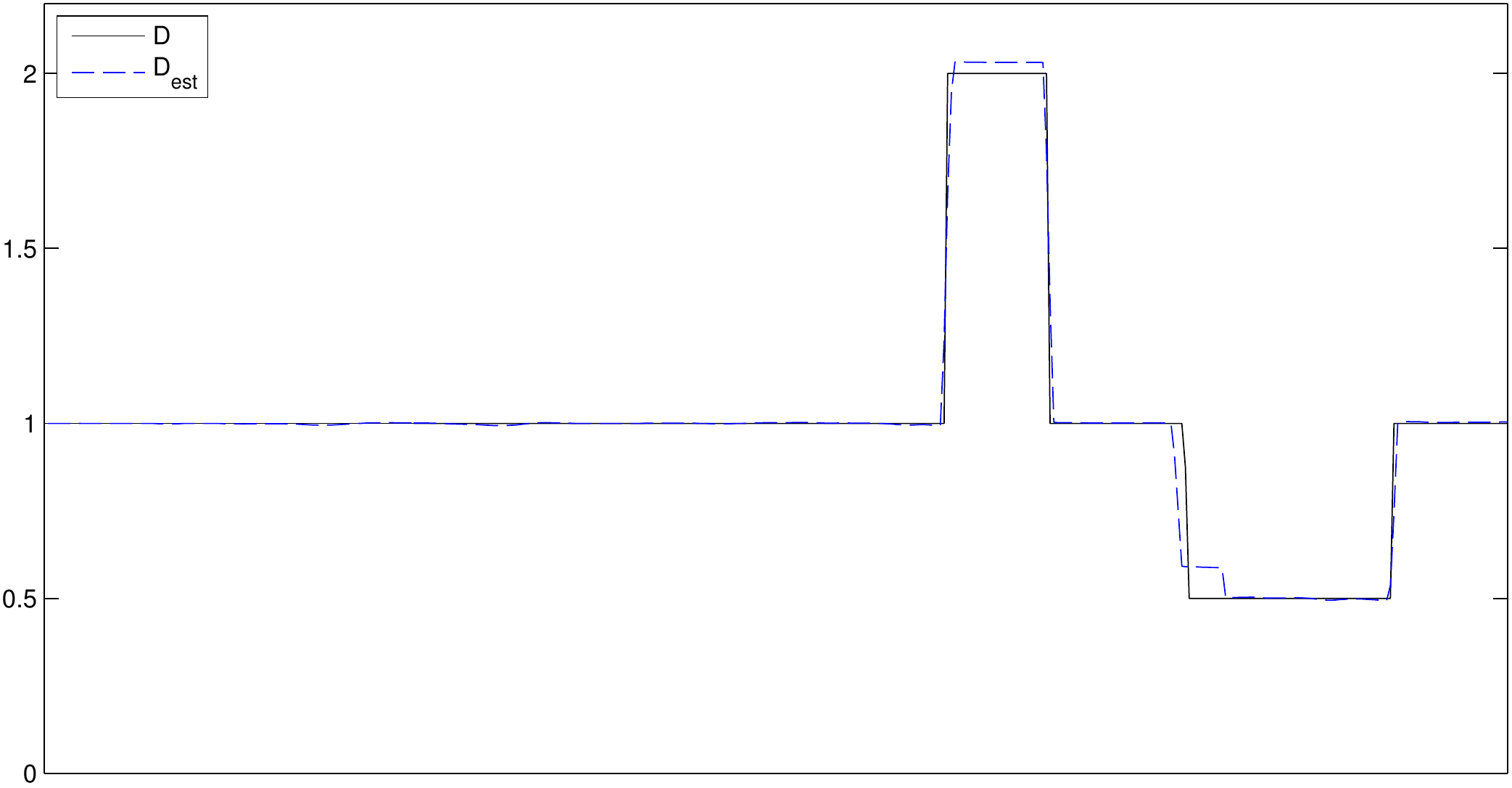}  \hspace{\baselineskip}		
		\includegraphics[width=0.29\textwidth]{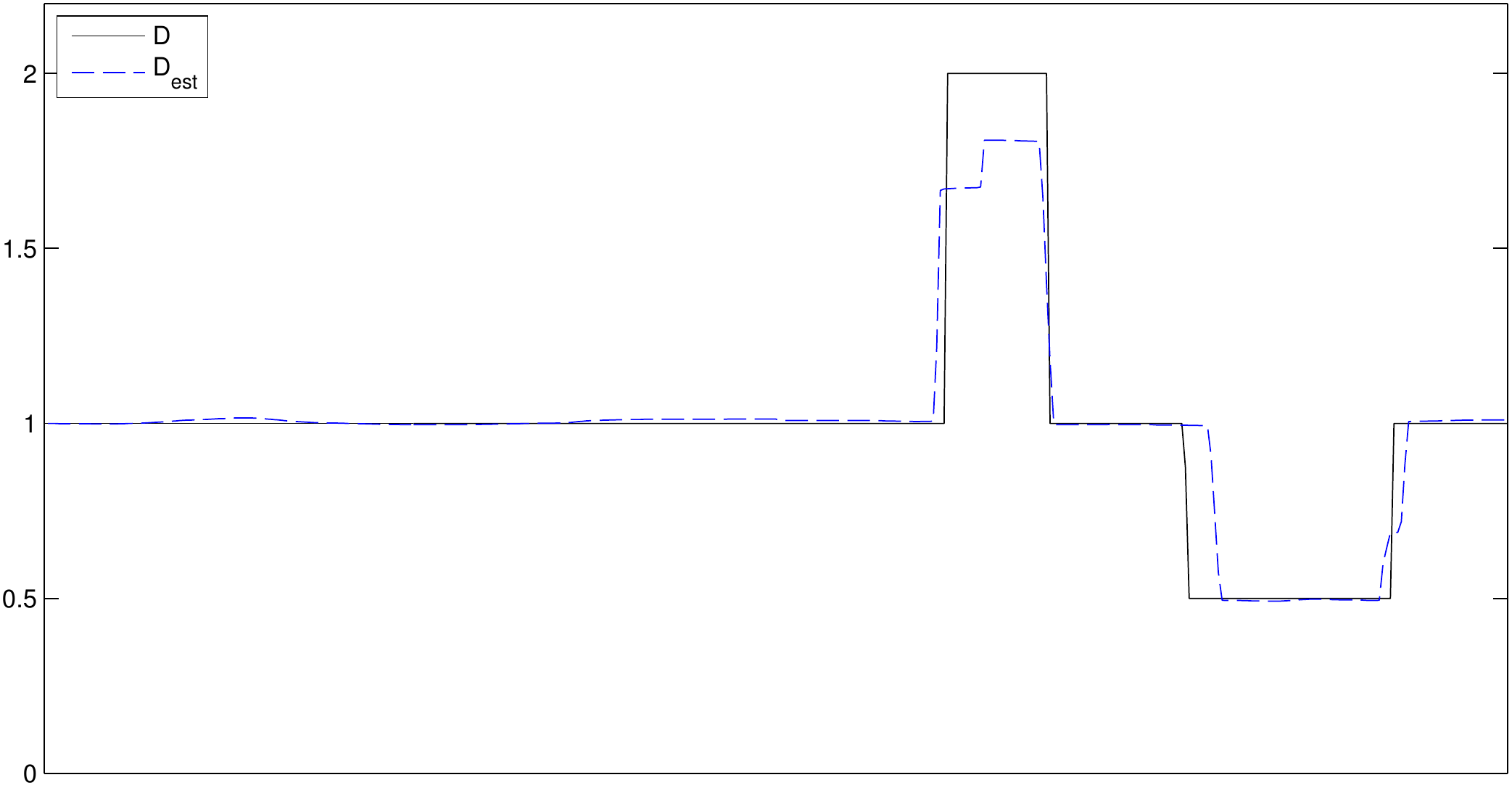}  \\ 					
	\caption[Reconstruction results]
	{ Reconstructions A. Columns: $0\%$, $0.1\%$ and $10\%$ noise. Rows: Estimated edge set, reconstructed parameters $\mu_{\text{est}}$, $D_{\text{est}}$ and error profiles for $\mu_{\text{est}},D_{\text{est}}$.  The error profiles show values of the true and reconstructed parameters along the image diagonals. The color axis in the images of the reconstructed parameters was fixed to the same range as the true parameters shown in Figure \ref{fig:setup_circles}. }		
	\label{fig:results_circles}	
\end{figure}

\begin{figure}[ht]
	\centering
		\includegraphics[width=0.29\textwidth]{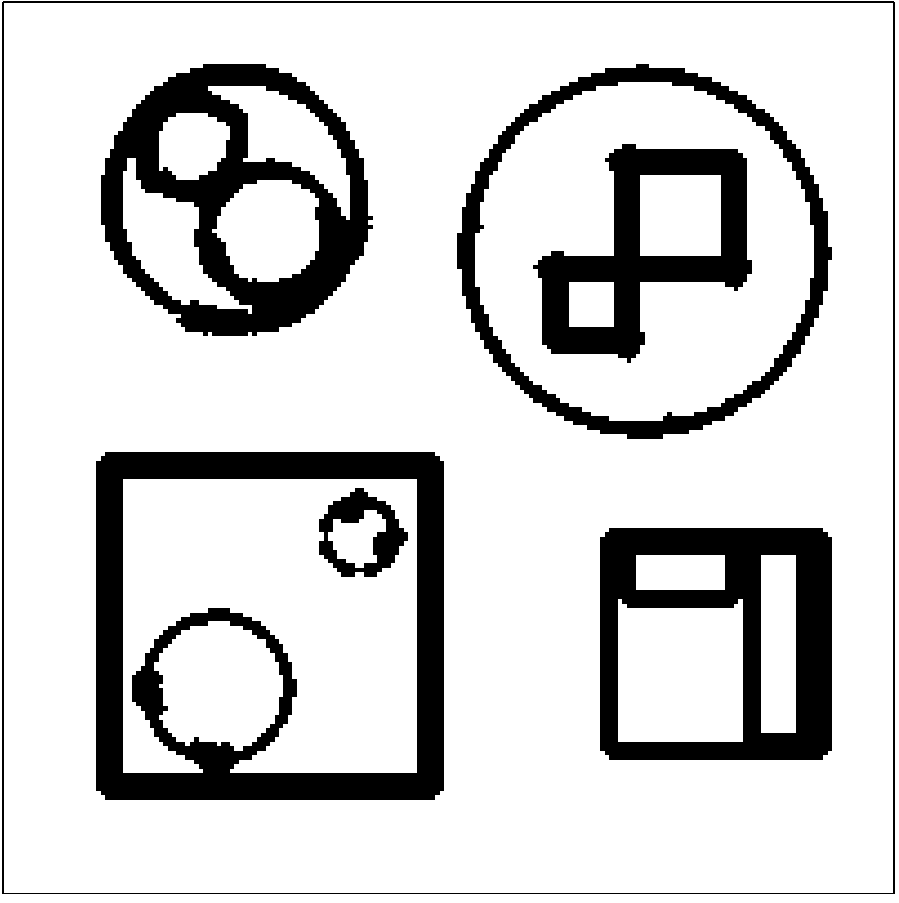}  \hspace{\baselineskip}
		\includegraphics[width=0.29\textwidth]{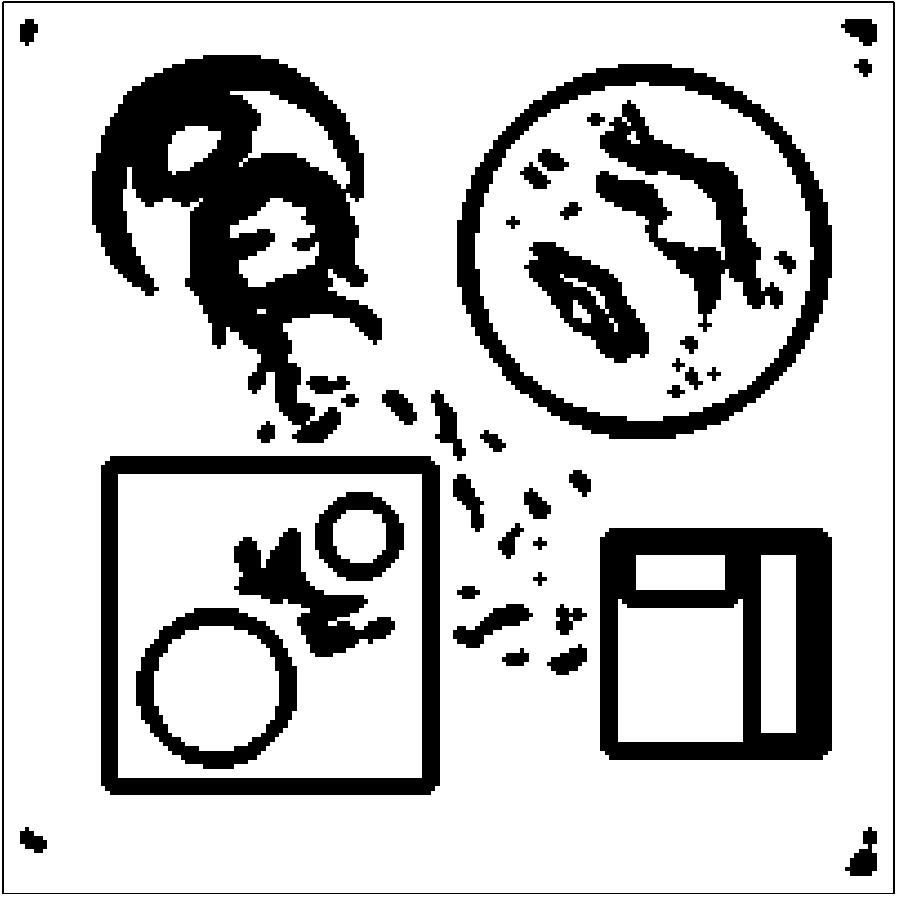}  \hspace{\baselineskip}		
		\includegraphics[width=0.29\textwidth]{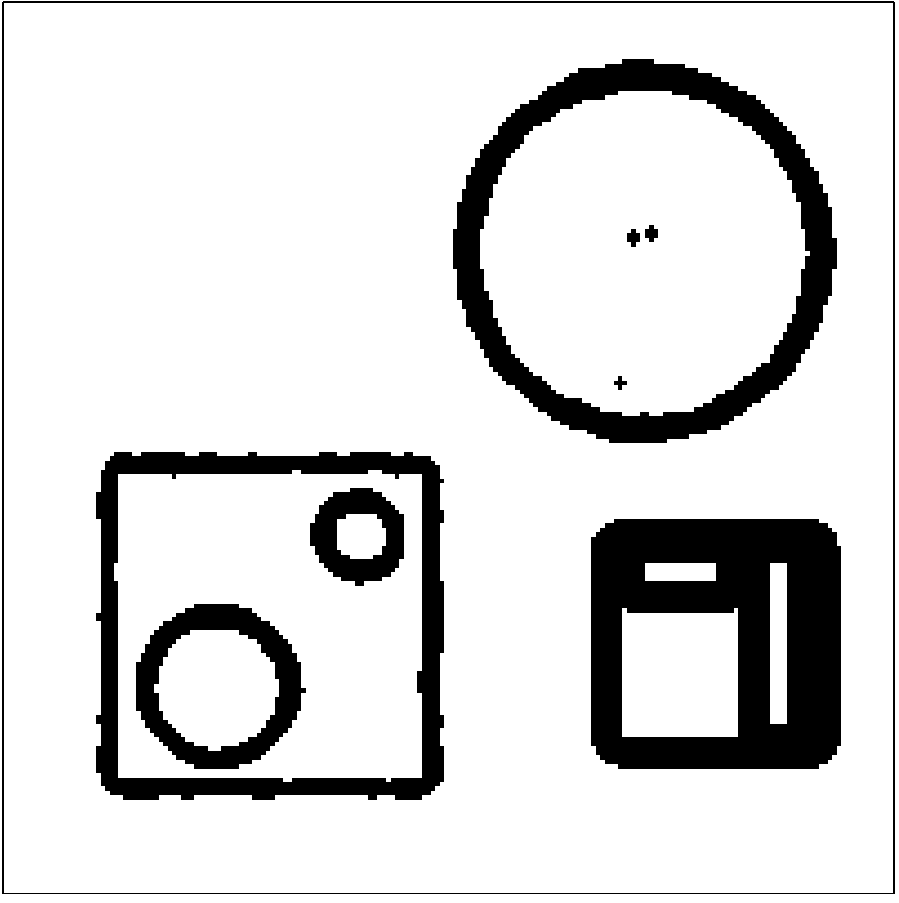}  \\ \vspace{\baselineskip}	
		\includegraphics[width=0.29\textwidth]{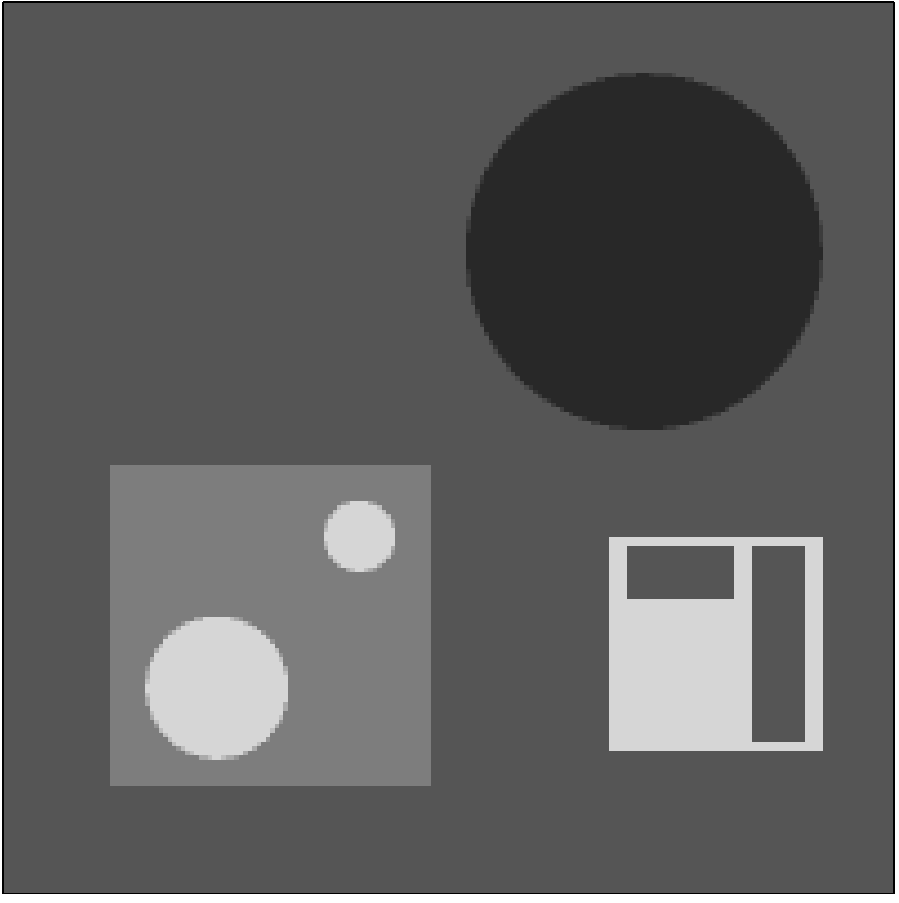}  \hspace{\baselineskip}
		\includegraphics[width=0.29\textwidth]{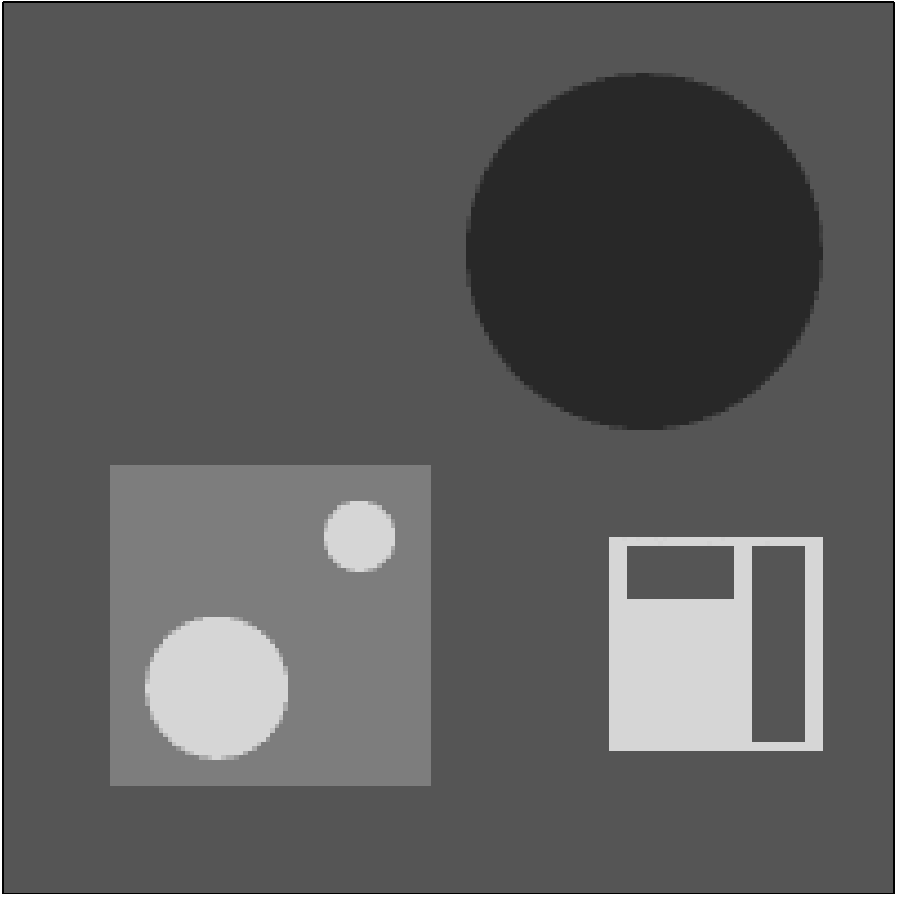}  \hspace{\baselineskip}		
		\includegraphics[width=0.29\textwidth]{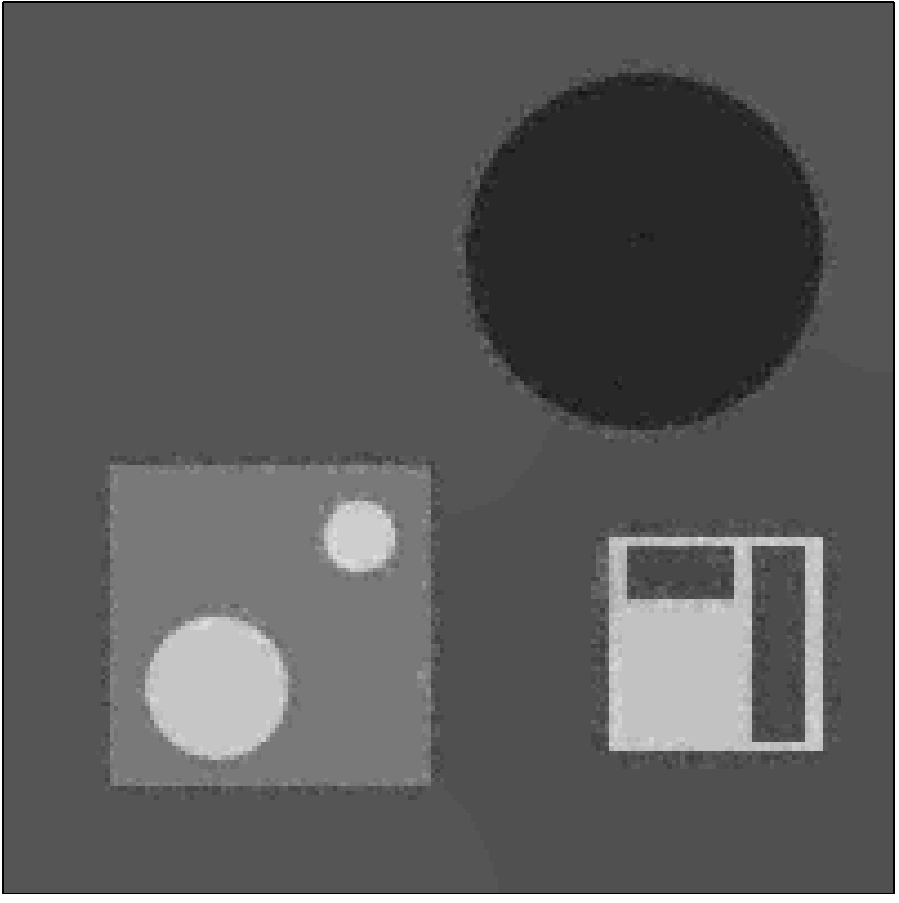}  \\ \vspace{\baselineskip}		
		\includegraphics[width=0.29\textwidth]{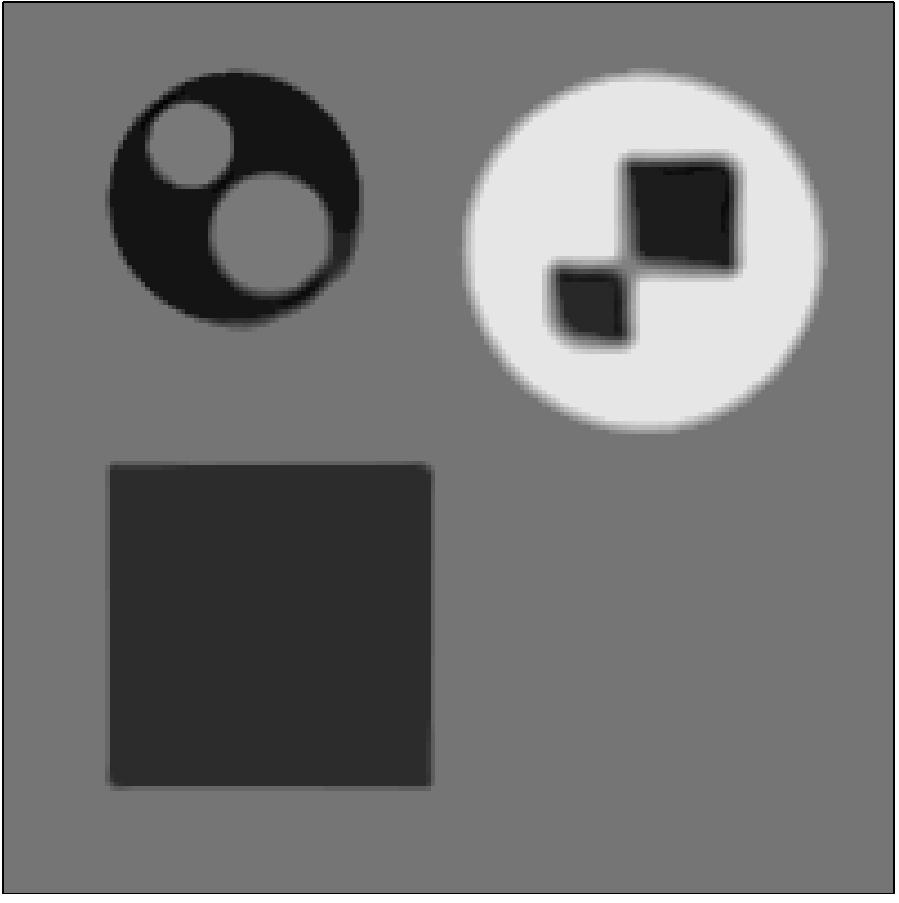}  \hspace{\baselineskip}
		\includegraphics[width=0.29\textwidth]{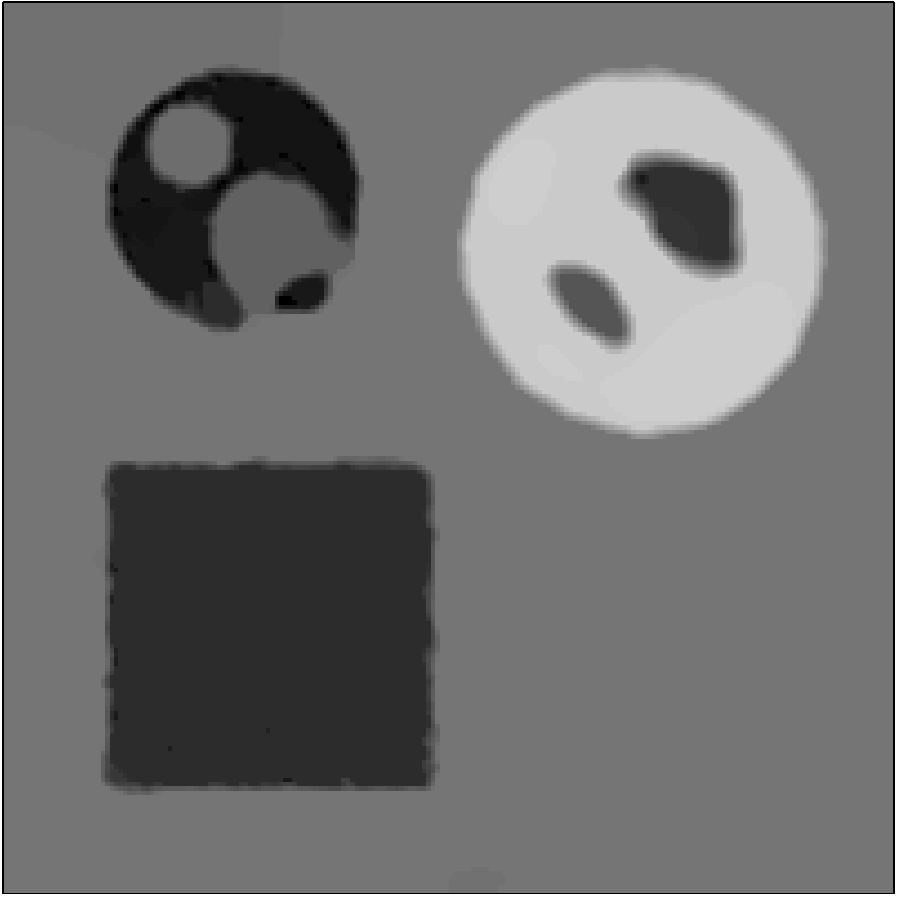}  \hspace{\baselineskip}		
		\includegraphics[width=0.29\textwidth]{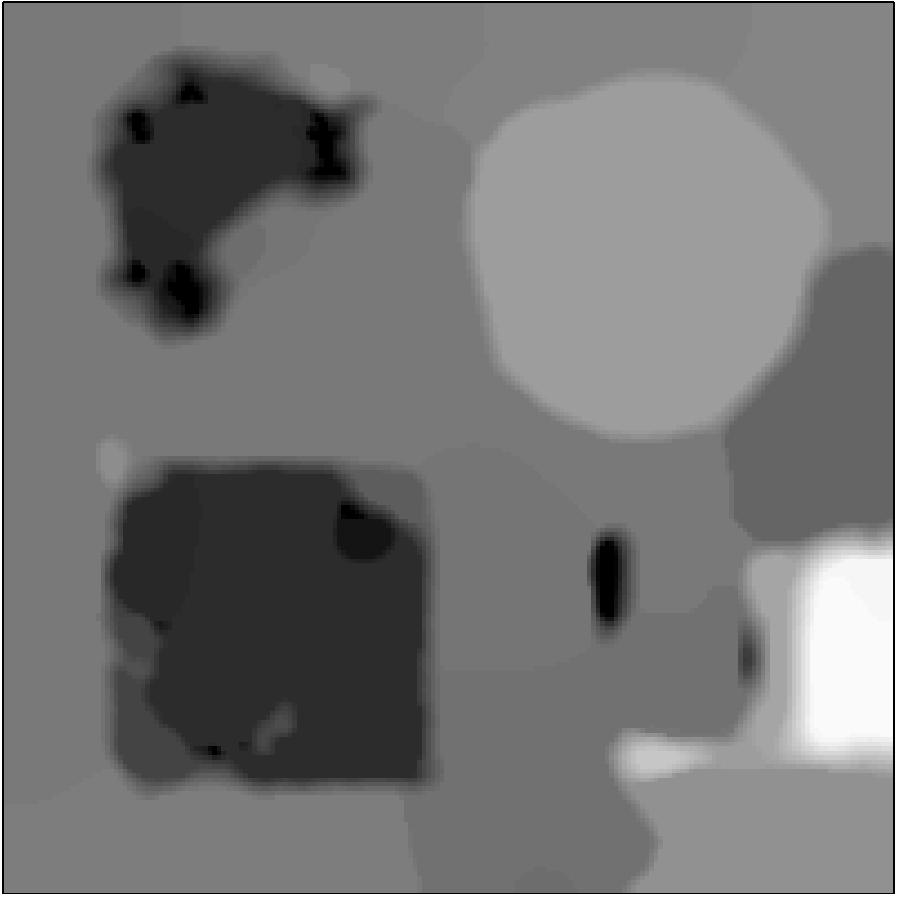}  \\ \vspace{\baselineskip}	
		\includegraphics[width=0.29\textwidth]{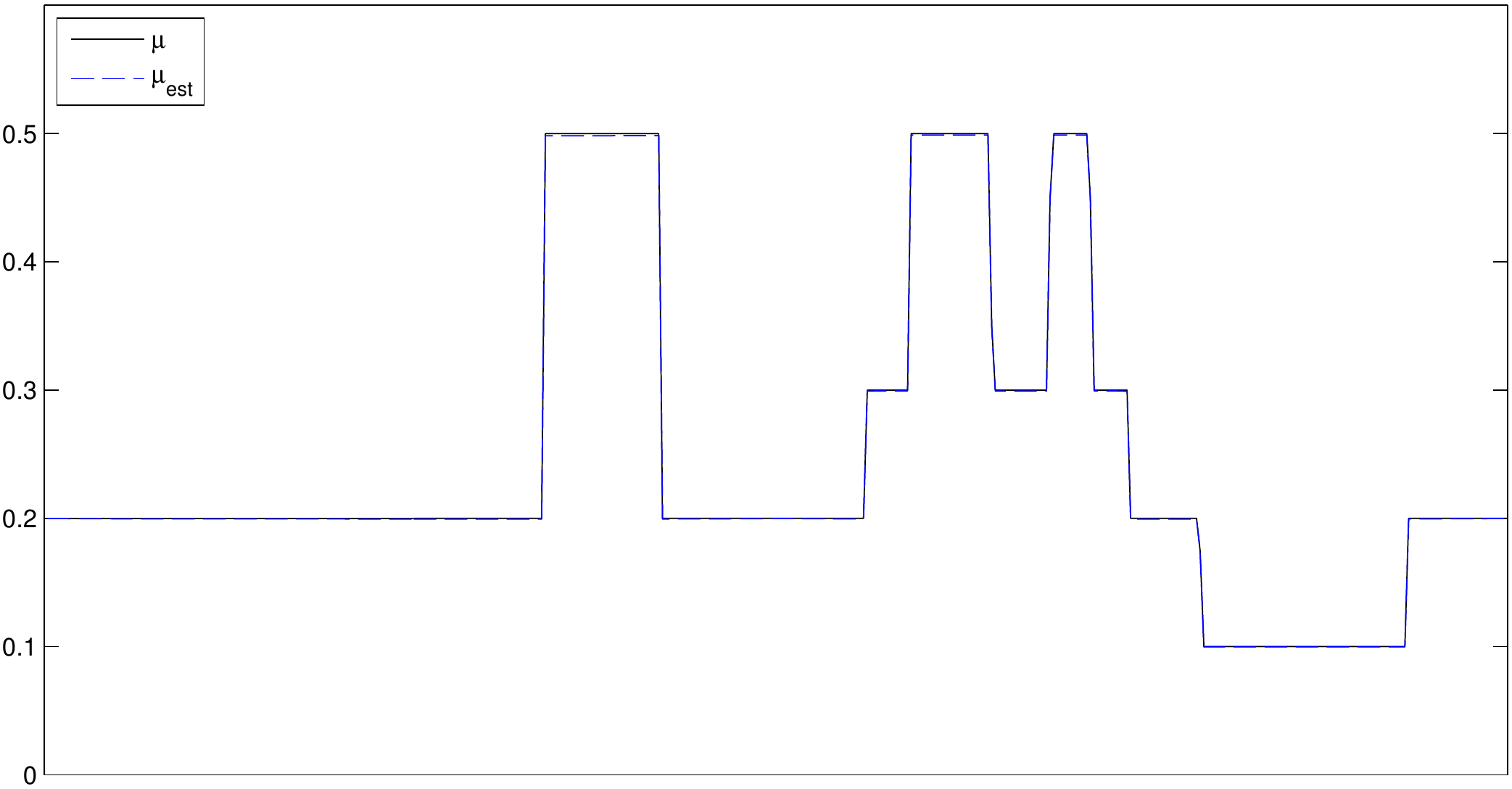}  \hspace{\baselineskip}
		\includegraphics[width=0.29\textwidth]{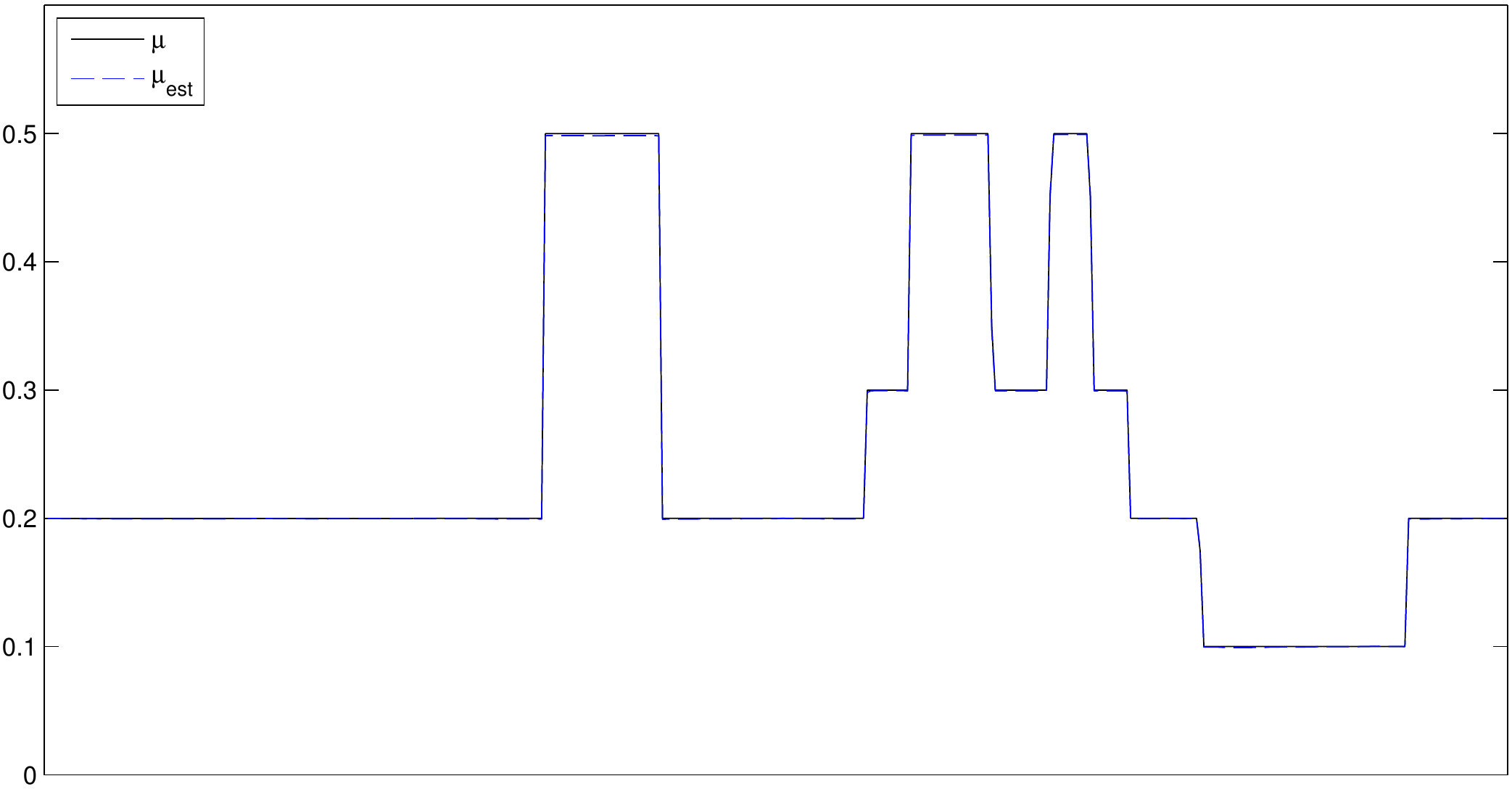}  \hspace{\baselineskip}		
		\includegraphics[width=0.29\textwidth]{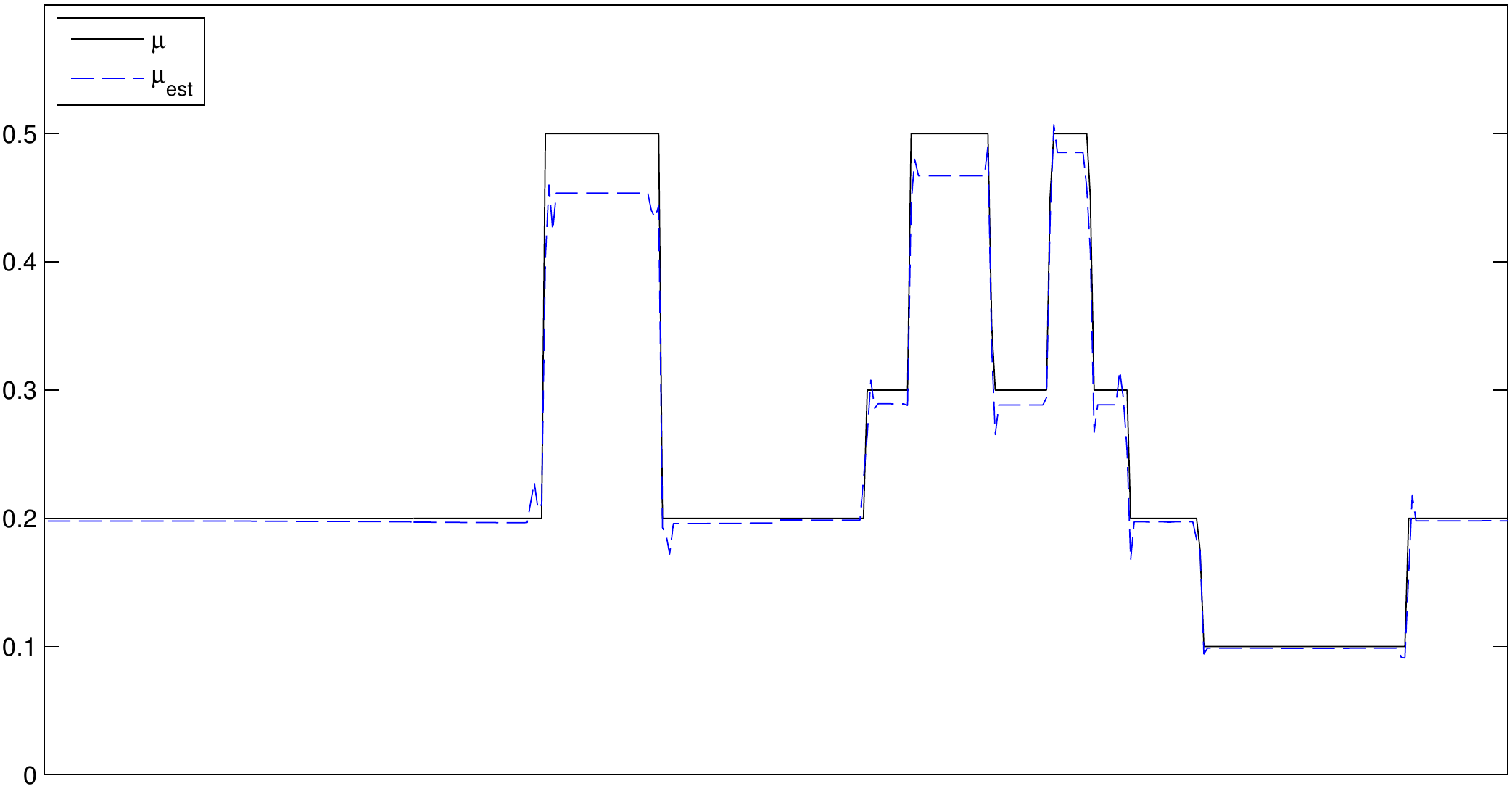}  \\ \vspace{\baselineskip}	
		\includegraphics[width=0.29\textwidth]{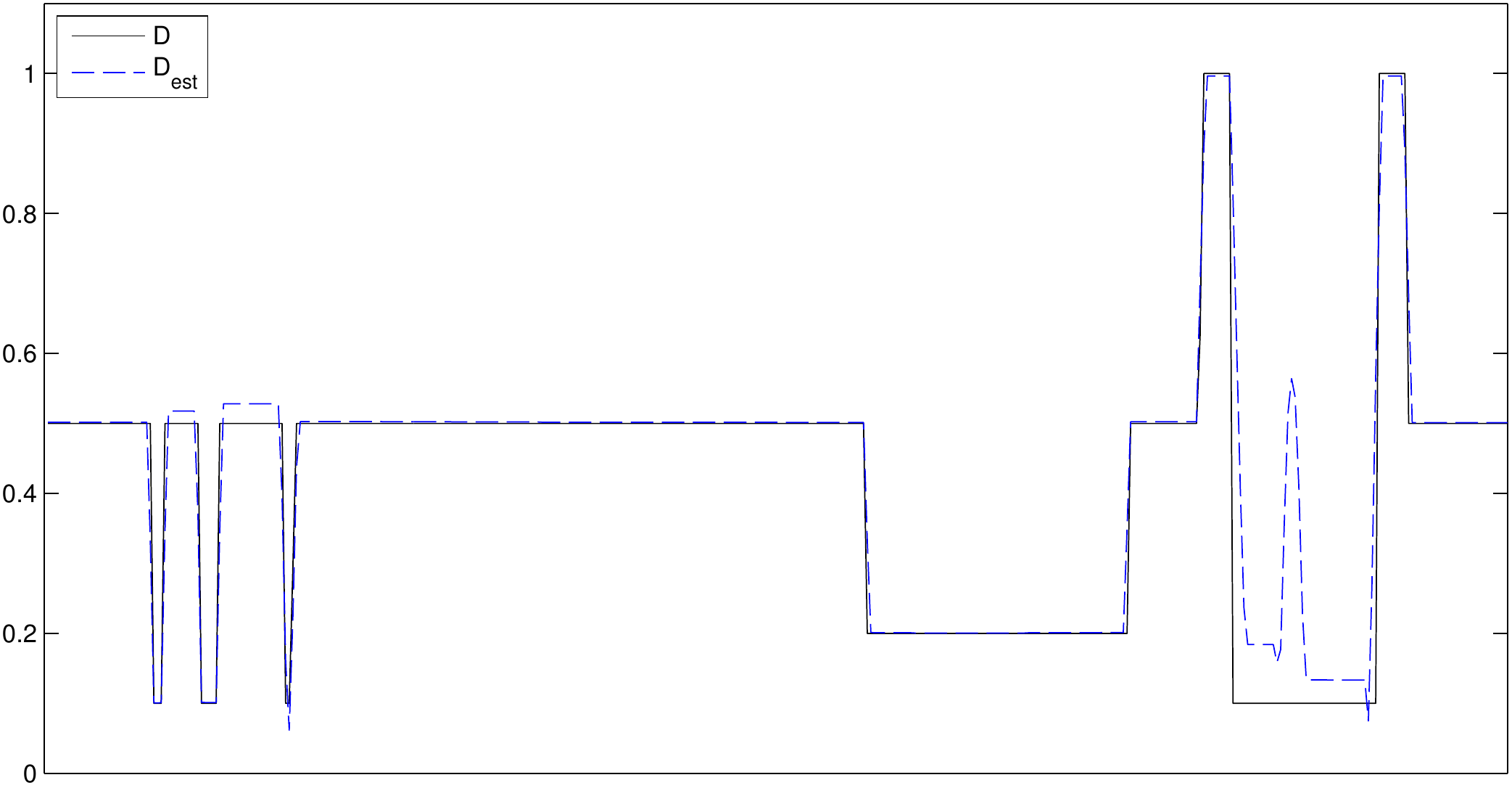}  \hspace{\baselineskip}
		\includegraphics[width=0.29\textwidth]{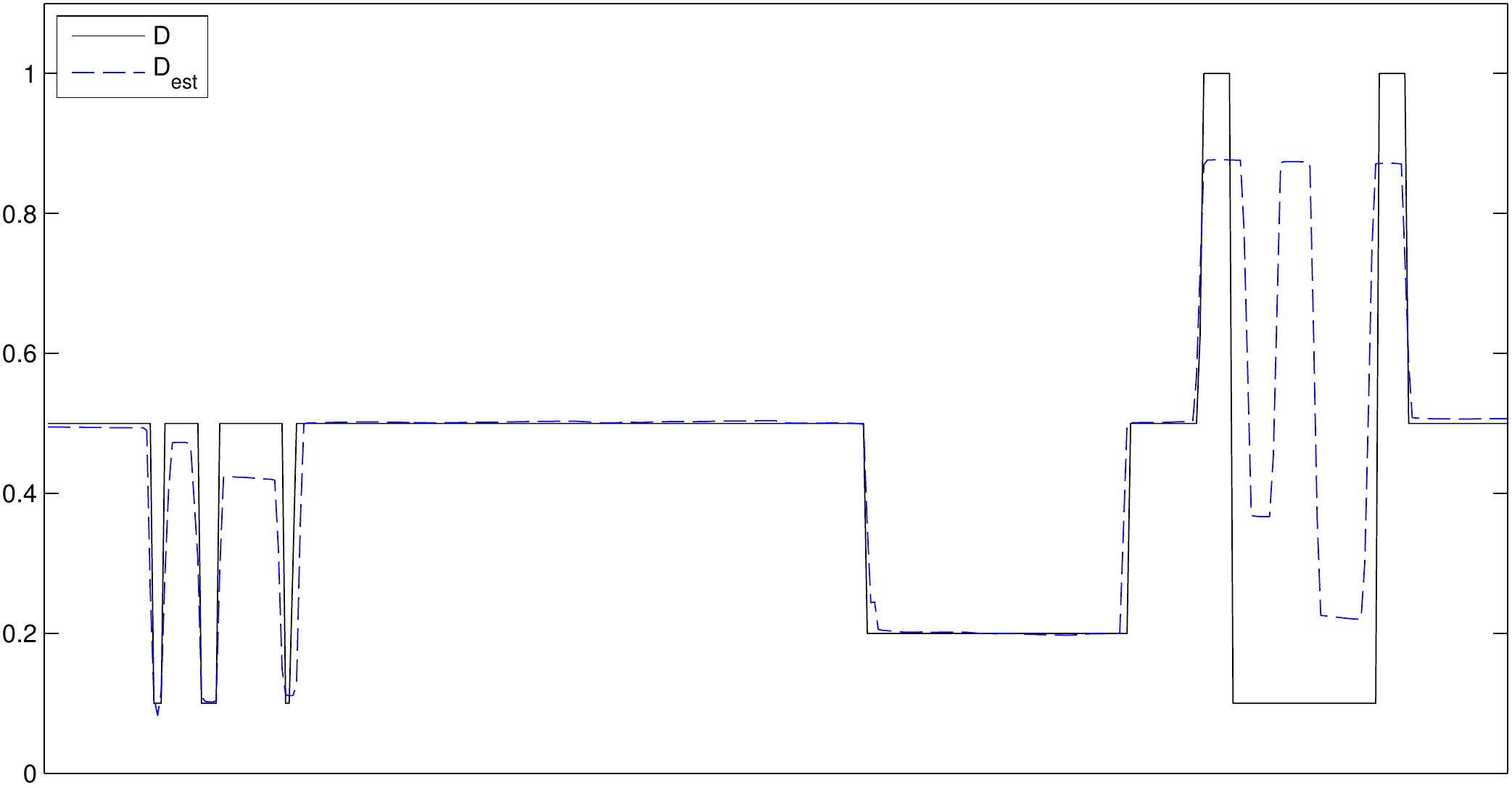}  \hspace{\baselineskip}		
		\includegraphics[width=0.29\textwidth]{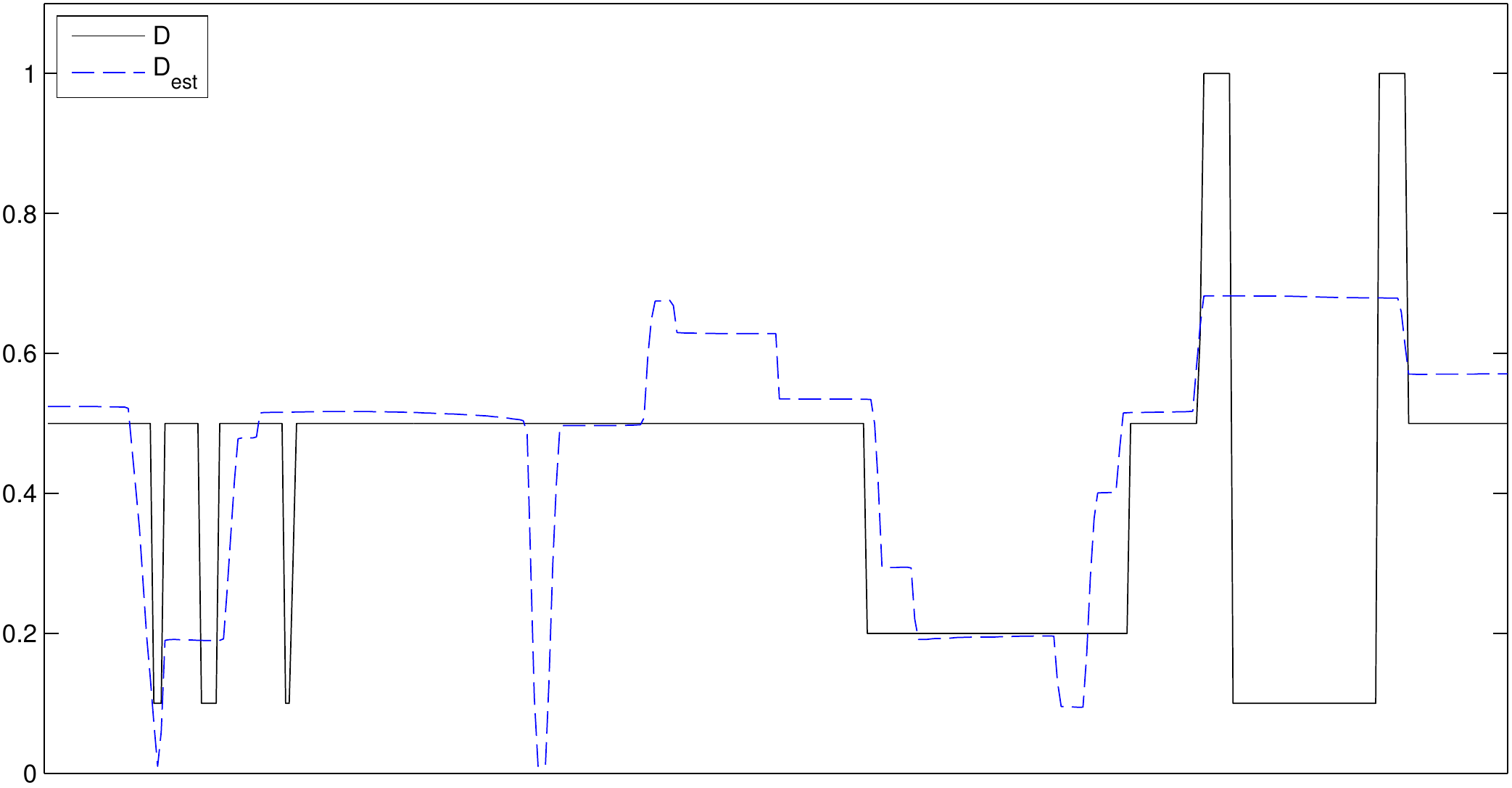}  \\
	\caption[Reconstruction results]
	{ Reconstructions B. Columns: $0\%$, $0.1\%$ and $10\%$ noise. Rows: Estimated edge set, reconstructed parameters $\mu_{\text{est}}$, $D_{\text{est}}$ and error profiles for $\mu_{\text{est}},D_{\text{est}}$.  The error profiles show valued of the true and reconstructed parameters along the image diagonals. The color axis in the images of the reconstructed parameters was fixed to the same range as the true parameters shown as in Figure \ref{fig:setup_rectangles}.  }	
	\label{fig:results_rectangles}	
\end{figure}

\begin{remark}
In some of the examples, minimization of \eqref{eq:qpat_ms_funct} using log-parameters, i.e., the mapping $\overline\E\colon (\log \mu, \log D) \mapsto \E(\mu,D)$ instead of $\E$, gave slightly better results. Note that this approach leads to a different linearization and therefore also a different system \eqref{eq:qpat_ms_system} to be solved in every step. However, to keep the presentation simple, we decided to use the functional as presented in Section \ref{sec:mumford_shah_parameter_detection} for our numerical experiments. 

We also tried to incorporate the Gr\"uneisen coefficient $\Gamma$ as an additional unknown into the reconstruction process, that is, to solve the problem \eqref{prob:problem3}. Unfortunately, this proved to be highly unstable, even if the initial edge set was detected perfectly.
\end{remark}

\section{Acknowledgements}
\label{sec:acknowledgements}

This work has been supported by the Austrian Science Fund (FWF) within the project FSP P26687-"Interdisciplinary Coupled Physic Imaging" and by the IK I059-N funded by the University of Vienna. 

\FloatBarrier

\appendix
\section{Special functions of bounded variation and the SBV-compactness theorem}
\label{sec:sbv_introduction}

This section briefly introduces the notion of $\SBV$-functions and their compactness theorem. For a more comprehensive presentation with proofs, see, e.g., \cite{AttButMic06}.

For a function $f \in L^1(\Omega)$ with distributional gradient $Df$, we define its \emph{total variation} by
\begin{equation*}
	|Df|_\Omega := \sup \{ \langle Df, \varphi \rangle \mid \varphi \in C^1_c(\Omega;\mathbb{R}^N), \norm{\varphi}_{L^\infty(\Omega;\mathbb{R}^N)} \leq 1 \}
\end{equation*}
The space $\BV(\Omega)$, consisting of all $L^1(\Omega)$-functions with finite total variation (i.e., of bounded variation), is a Banach space with the norm
\begin{equation*}
	\norm{f}_{\BV(\Omega)} := \norm{f}_{L^1(\Omega)} + |Df|_\Omega.
\end{equation*}

Note that by the Riesz-Markov representation theorem, functions $f \in L^1(\Omega)$ are of bounded variation if and only if $Df$ is a \emph{finite vector Radon measure}. The measure $Df$ can be decomposed into three parts
\begin{equation*}
	Df = D^a f + D^j f + D^c  f.
\end{equation*}
$D^a f$ is the part of $Df$ that is \emph{absolutely continuous} with respect to the Lebesgue measure $\L^N$, i.e., 
\begin{equation*}
	D^a f = \nabla f \,\L^N
\end{equation*} 
for some integrable density function $\nabla f$. The \emph{jump part} $D^j f$ is concentrated on the jump (or approximate discontinuity) set $S(f)$ defined by
\begin{equation*}
	S(f):=\{ x \in \Omega \mid f^-(x) < f^+(x) \},
\end{equation*}
where, denoting with $B_\rho(x)$ the ball centered at $x$ with radius $\rho$,
\begin{equation*}
\begin{aligned}
	f^+(x) &:= \inf \left\{ t \in \mathbb{R} \mid \lim_{\rho \to 0} \frac{\L^N(\{y \in B_\rho(x) \mid f(y) > t \})}{\L^N(B_\rho(x))}=0 \right\} \\
	f^-(x) &:= \sup \left\{ t \in \mathbb{R} \mid \lim_{\rho \to 0} \frac{\L^N(\{y \in B_\rho(x) \mid f(y) < t \})}{\L^N(B_\rho(x))}=0 \right\}.
\end{aligned}
\end{equation*}
Furthermore, for $H^{N-1}$-almost all $x \in S(f)$, there exists a \emph{unit normal vector} $\nu(x)$ and we have
\begin{equation*}
	D^j f = (f^+ - f^-)\, \nu \, H^{N-1}|_{S(f)}.
\end{equation*}
The remaining \emph{Cantor part} $D^c f$ is concentrated on a subset of $\Omega \setminus S(f)$ with intermediate Hausdorff dimension between $N-1$ and $N$.

A function $f \in \BV(\Omega)$ is a \emph{special function of bounded variation} (or $f \in \SBV(\Omega)$) if $D^c f=0$, that is,
\begin{equation*}
	Df = \nabla f \,\L^N + (f^+ - f^-)\, \nu \, H^{N-1}|_{S(f)}.
\end{equation*}

Furthermore, we have the following \emph{compactness theorem} due to Ambrosio \cite{Amb89a}:
\begin{theorem}
	Let $f_n$ be a sequence in $\SBV(\Omega)$ with
	\begin{equation*}
		\norm{f_n}_{L^\infty(\Omega)} + \int_\Omega |\nabla f_n|^p \,dx + H^{N-1}[S(f_n)] \leq M < \infty 
	\end{equation*}	
	 for all $n \in \mathbb{N}$ and some $p > 1, M > 1$. Then there exist a subsequence $\left( f_{n_k} \right)_{k \in \mathbb{N}}$ and $f \in \SBV(\Omega)$ with
	 \begin{equation*}
	 	\begin{aligned}
	 			f_{n_k} &\to f \text{ strongly in } L^1_{\text{loc}}(\Omega) \\
				\nabla f_{n_k} &\rightharpoonup \nabla f  \text{ weakly in } L^p(\Omega;\mathbb{R}^N) \\
				H^{N-1}[f] &\leq \liminf_{k \to \infty} H^{N-1}[f_{n_k}]. 
	 	\end{aligned}
	 \end{equation*}
\end{theorem}

\def\cprime{$'$}
  \providecommand{\noopsort}[1]{}\def\ocirc#1{\ifmmode\setbox0=\hbox{$#1$}\dimen0=\ht0
  \advance\dimen0 by1pt\rlap{\hbox to\wd0{\hss\raise\dimen0
  \hbox{\hskip.2em$\scriptscriptstyle\circ$}\hss}}#1\else {\accent"17 #1}\fi}

\end{document}